\documentclass{amsart}
\usepackage{amsmath,amssymb,dsfont}

\usepackage{graphicx,color}
\usepackage{pstricks, pst-plot, pst-node, pst-grad, pst-math}
\usepackage{pst-plot}
\usepackage{todonotes}
\usepackage[utf8]{inputenc}
\usepackage[T1]{fontenc}
\usepackage{lmodern}
\newtheorem{theorem}{Theorem}[section]
\newtheorem{lemma}[theorem]{Lemma}

\newtheorem{corollary}[theorem]{Corollary}

\theoremstyle{definition}

\newtheorem{remark}[theorem]{Remark}
\newtheorem{assumption}[theorem]{Assumption}

\numberwithin{equation}{section}

\renewcommand{\vec}[1]{\left( \!\! \begin{array}{c} #1 \end{array} \!\! \right)}

\newcommand{\rd}{{\,\rm d}}
\newcommand{\e}{{\rm e}}

\newcommand{\N}{{\mathbb N}}

\newcommand{\R}{{\mathbb R}}
\newcommand{\C}{{\mathbb C}}

\newcommand{\dist}{\mathrm{dist}}
\newcommand\re{\mathrm{Re}}
\newcommand\im{\mathrm{Im}}
\newcommand\I{\mathrm{i}}
\newcommand{\beqnt}{\begin{equation*}}
\newcommand{\eeqnt}{\end{equation*}}
\newcommand{\set}[2]{\{#1 : #2 \}}

\newcommand{\sgn}{\operatorname{sgn}}

\DeclareMathOperator{\Tr}{Tr}

\DeclareMathOperator{\supp}{supp}
\DeclareMathOperator{\Det}{Det}

\begin{document}

\title[Dirac and Schr\"odinger operators with complex potentials]{Eigenvalue bounds for Dirac and fractional Schr\"odinger operators with complex potentials}

\author{Jean-Claude Cuenin}
\address{Mathematisches Institut, Ludwig-Maximilians-Universit\"at M\"unchen, 80333 Munich, Germany}
\email{cuenin@math.lmu.de}

\begin{abstract}
We prove Lieb-Thirring type bounds for fractional Schr\"odinger operators and Dirac operators with complex-valued potentials. The main new ingredient is a resolvent bound in Schatten spaces for the unperturbed operator, in the spirit of Frank and Sabin \cite{FrankSabin14}. 

\end{abstract}

\maketitle

%\tableofcontents

\section{Introduction}
Many recent publications have dealt with eigenvalue bounds for non-selfadjoint perturbations of operators from mathematical physics, see for example
\cite{AAD01,FrLaLiSe06,LaSa09,BorichevEtAl2009,Sa10,Frank11,DemuthEtAl2009,DemuthEtAl2013,Sambou2014,Dubuisson2014,Dubuisson2014frac,FrankSabin14}. 
One of the most common approaches, initiated in \cite{BorichevEtAl2009}, is to regard the eigenvalues as zeros of a holomorphic function (a regularized determinant) and then use function-theoretic arguments related to Jensen's identity to estimate sums of eigenvalues. Taking the Schr\"odinger operator $-\Delta+V$ as an example, eigenvalues could a priori accumulate at any point in $[0,\infty)$.
A typical result of \cite{FrankSabin14} is that for any sequence $\{z_j\}_j$ of eigenvalues accumulating to a point $\lambda\in (0,\infty)$ it holds 
that $\{\dist(z_j,[0,\infty))\}_j\in l^{1}$, provided $V\in L^q(\R^d)$, with $d/2<q\leq (d+1)/2$. This is an improvement of earlier results of \cite{DemuthEtAl2009,DemuthEtAl2013} where it was shown that such a sequence is in $l^{q+\epsilon}$ for certain $q>d/2$. Additionally, the latter estimates require a lower bound on the real part of $V$ or an estimate of the numerical range of $-\Delta+V$.
In this paper we prove that $\{\dist(z_j,[0,\infty))\}_j\in l^{1}$ for eigenvalues of $H_0+V$ where $H_0$ is either a fractional Laplacian, fractional Bessel or a (massless or massive) Dirac operator. This restriction is somewhat arbitrary, but we found our particular choice to be a reasonable generalization of the results of \cite[Sect. 4 and 6]{FrankSabin14}. With this in mind, we have made an effort to state the key estimates, Lemma~\ref{lemma resolvent estimate low frequency} and especially Lemma~\ref{lemma schatten estimate low frequency}, in greater generality than needed. The techniques for proving these estimates are standard in harmonic analysis (complex interpolation, stationary phase), and the proofs bear close resemblance to the proof of the Stein-Tomas restriction theorem. The two lemmas are used to prove Theorem \ref{theorem uniform Sobolev inequality fractional Laplacian} and Theorem \ref{main theorem schatten norms}. The first is an analogue of the uniform Sobolev inequality due to Kenig, Ruiz and Sogge \cite[Thm. 2.3]{KRS} stating that for\footnote{In fact, both \eqref{KRS} and \eqref{FrankSabin Schatten bound} hold also in $d=2$ if one excludes an endpoint. Moreover, the general case $|z|>0$ is obtained by scaling.} $d\geq 3$ and $2d/(d+2)\leq p\leq 2(d+1)/(d+3)$, and for all $z\in\C$ with $|z|= 1$, the estimate
\begin{align}\label{KRS}
\|(-\Delta-z)^{-1}\|_{L^p(\R^d)\to L^{p'}(\R^d)}\leq C
\end{align}
holds with a constant independent of $z$. 
%We use the convention that $L^p=L^p(\R^d)$, unless stated otherwise.
The second lemma is an analogue of the closely related uniform resolvent estimate in Schatten spaces due to Frank and Sabin \cite[Thm. 12]{FrankSabin14}, which we state for $d\geq 3$: For $d/2\leq q\leq (d+1)/2$ and for all $z\in\C$ with $|z|=1$, one has the estimate
\begin{align}\label{FrankSabin Schatten bound}
\|W_1(-\Delta-z)^{-1}W_2\|_{\mathfrak{S}^{q(d-1)/(d-q)}(L^2(\R^d))}\leq C\|W_1\|_{L^{2q}(\R^d)}\|W_2\|_{L^{2q}(\R^d)}
\end{align}
with a constant independent of $z$. Here, $\mathfrak{S}^p(L^2(X))$, $1\leq p<\infty$, denote the Schatten spaces
\begin{align*}
\mathfrak{S}^p(L^2(X))=\set{T\in\mathfrak{S}^{\infty}(L^2(X))}{\Tr(T^*T)^{p/2}<\infty},
\end{align*}
and $\mathfrak{S}^{\infty}(L^2(X)$ is the space of compact operators on $L^2(X)$ with the operator norm topology.
For $1\leq p<q<\infty$, we have the continuous embeddings
\begin{align}\label{embeddings Schatten spaces}
\mathfrak{S}^p(L^2(X))\subset \mathfrak{S}^q(L^2(X))\subset \mathfrak{S}^{\infty}(L^2(X)).
\end{align}
In view of \eqref{embeddings Schatten spaces}, the bound \eqref{FrankSabin Schatten bound} implies \eqref{KRS} by a duality argument. At the same time, the proof of \eqref{FrankSabin Schatten bound} relies on the kernel bounds of \cite{KRS} for complex powers of the resolvent. These bounds were obtained by using an explicit formula for the Fourier transform of the symbol. In our case, no explicit formulas are available, and the proof has to be more flexible. Here, we rely on a factorization of the symbol of the resolvent that is often encountered in microlocal analysis. In particular, in the context of $L^p$ estimates it was used in \cite{MR1168960}. Although the method is essentially known, our resolvent estimates seem to be new. The range of permitted exponents in Theorem \ref{main theorem schatten norms} is larger than what is stated in \cite{FrankSabin14}, and it settles the question whether \ref{FrankSabin Schatten bound} can be extended from $q\geq 4/3$ down to $q>1$ in two dimensions (see Remark \ref{rem. question remark 8 in FS}). This remark also yields an improvement in the range of $q$ in Theorems 13 and 17 of \cite{FrankSabin14} in dimension two.

The outline of the paper is as follows. In Section 2 we state our assumptions and the main result (Theorem \ref{Main theorem}) about the accumulation rate of complex eigenvalues to a non-critical point $\lambda\in\sigma(H_0)$ and provide a short proof for the simplified case of compactly supported potentials. In Section 3 we prove uniform $L^p\to L^{p'}$ resolvent estimates for the free resolvent of the fractional Schr\"odinger and Dirac operator (Theorem \ref{theorem uniform Sobolev inequality fractional Laplacian}), both with and without mass. In Section 4 we prove analogous resolvent estimates in Schatten spaces (Theorem \ref{main theorem schatten norms}), which is our main technical result. In Section 5 we prove Theorem \ref{Main theorem} by combining the resolvent estimates with function-theoretic arguments. Section 6 contains new results about the location and the distribution of discrete eigenvalues of $H_0+V$. The paper concludes with an appendix that contains some technical details related to the proof of Lemma \ref{lemma schatten estimate low frequency} and a convexity lemma that we need to interpolate estimates for the sandwiched resolvent in Schatten spaces.

Notation: In the following we set $X=\R^d$ or $X=\R^d\otimes\C^n$ (in the case of the Dirac operator). In the latter case, the norm in $\C^n$ is fixed throughout the article. The resolvent set of a selfadjoint operator $H_0$ on $L^2(X)$ is denoted by $\rho(H_0)$.  For $z\in \rho(H_0)$ we write $R_0(z)=(H_0-z)^{-1}$ for the resolvent operator. The discrete spectrum $\sigma_{\rm d}(H)$ of a closed operator $H$ is the set of isolated eigenvalues of $H$ of finite algebraic multiplicity. The spaces $L^p(X)\cap L^q(X)$ and $L^p(X)+L^q(X)$ are equipped with the norms $\|f\|_{L^p\cap L^q}=\max\{\|f\|_{L^p},\|f\|_{L^q}\}$ and $\|f\|_{L^p+ L^q}=\inf_{f=f_1+f_2}(\|f_1\|_{L^p}+\|f_2\|_{L^q})$, respectively. Here and in the following, we omit the space $X$ from the notation of the $L^p$ norms. 
For a function $T:\R^d\to\R$ we denote the corresponding Fourier multiplier by $T(D)$, i.e.\ $T(D)f:=\mathcal{F}^{-1}(T\widehat{f})$, where $\mathcal{F}$ denotes the Fourier transformation. We also use the symbol $\mathcal{F}_1$ to denote the partial Fourier transform in the first variable. Constants are generic and are denoted by $C$ or by $C(K)$ if their dependence on a parameter $K$ is emphasized. The dependence on the dimension and on the exponents of the $L^p$ spaces involved will usually not be made explicit. We occasionally use the notation $a\approx b$ to indicate that $C^{-1}b\leq a\leq C b$ for some constant $C>0$.
A smooth compactly supported function $\chi$ will sometimes be referred to simply as a bump function.

\section{Main results}

\begin{assumption}\label{unperturbed}
Let $d\geq 2$. The unperturbed Hamiltonian is $H_0=T(D)$ on $L^2(X)$, where $T(D)$ is one of the following kinetic energies
\begin{itemize}
\item The fractional Laplacian $T(D)=(-\Delta)^{s/2}$,
\item The fractional Bessel operator $T(D)=(1-\Delta)^{s/2}-1$, 
\item The massless Dirac operator $T(D)=\sum_{j=1}^d\alpha_j(-\I\nabla_j)\equiv \mathcal{D}_0$, 
\item The massive Dirac operator\footnote{Here, $\alpha_j$, $j=1,\ldots,n$ and $\beta=\alpha_{n+1}$ are Hermitian matrices satisfying the Clifford relations $\alpha_i\alpha_j+\alpha_j\alpha_i=2\delta_{ij}I_n$, where $n$ may be chosen as $n=2^{d/2}$ if $d$ is even and $n=2^{(d+1)/2}$ if $d$ is odd. We will suppress the identity $I_n$ in the following.} $T(D)=\sum_{j=1}^d\alpha_j(-\I\nabla_j)+\beta\equiv \mathcal{D}_1$.
\end{itemize}
We denote by $\Lambda_c(H_0)$ the set of critical values of $T(\cdot)$.
For the Dirac operators $\mathcal{D}_0$ and $\mathcal{D}_1$ we consider the the symbol corresponding to their eigenvalues,
$\lambda_{\pm}(\xi)=\pm|\xi|$ and $\lambda_{\pm}(\xi)=\pm(1+|\xi|^2)^{1/2}$, respectively. Hence,
\begin{align*}
\Lambda_c((-\Delta)^{s/2})=
\begin{cases}
\{0\}\quad&\mbox{if } s>1,\\
\emptyset \quad&\mbox{if } s\leq 1,
\end{cases}
\quad \Lambda_c((1-\Delta)^{s/2}-1)=\{0\},
\end{align*}
and 
\begin{align*}
\Lambda_c(\mathcal{D}_0)=\emptyset,
\quad \Lambda_c(\mathcal{D}_1)&=\{1,-1\}.
\end{align*}
We adopt the convention that $s=1$ when $H_0=\mathcal{D}_0$ or $H_0=\mathcal{D}_1$. To limit case-by-case arguments, we will assume that $0<s<d$. 
\end{assumption}

\begin{assumption}\label{potential}
Let $0<s<d$ be fixed, and let $V$ be a complex-valued\footnote{In the case where $H_0$ is a Dirac operator ($\mathcal{D}_0$ or $\mathcal{D}_1$) we allow $V$ to be a (generally non-hermitian) matrix-valued potential.} potential. Assume $s$ and $V$ satisfy one of the following assumptions.
\begin{itemize}
\item[a)] $s\geq 2d/(d+1)$ and $V\in L^q(X)$, where $d/s\leq q\leq (d+1)/2$. 
\item[b)] $s< 2d/(d+1)$ and $V\in L^{d/s}(X)\cap L^{(d+1)/2}(X)$.
\end{itemize}
\end{assumption}

\begin{theorem}\label{Main theorem}
Let $H_0=T(D)$, and let $V$ be a complex-valued potential such that Assumptions \ref{unperturbed}--\ref{potential} hold. If $K$ is a compact subset of $\C\setminus \Lambda_c(H_0)$, then 
\begin{align}\label{sum of eigenvalues converges}
\sum_{z\in \sigma_{\rm d}(H_0+V)\cap K}\dist(z,\sigma(H_0))\leq C(K,V).
\end{align}
In the above sum, each eigenvalue $z$ is counted according to its algebraic multiplicity.
Moreover, there exists $C(K)$ such that if $\|V\|\leq C(K)$, then $\sigma(H_0+V)\cap K=\emptyset$. Here, 
\begin{align}\label{def. of ||V||}
\|V\|:=\begin{cases}
\|V\|_{L^q}\quad&\mbox{in case a)},\\
\|V\|_{L^{d/s}\cap L^{(d+1)/2}}\quad&\mbox{in case b)}.
\end{cases}
\end{align}
\end{theorem}

\begin{remark}
In case a), when $d/s<q\leq (d+1)/2$, there is an ''effective'' bound in terms of the $L^q$-norm of $V$, i.e.\ the right hand side of \eqref{sum of eigenvalues converges} may be replaced by $C(K,\|V\|_{L^q})$, see Theorem \ref{thm eigenvalue sums big s} in Section \ref{section further results}. There we also prove versions of \eqref{sum of eigenvalues converges} where the sum is taken over all eigenvalues, not only those in $K$. The price to pay for this generalization is that one has to insert a weight which may become zero at the critical values $\Lambda_c(H_0)$ or at infinity.
\end{remark}
 
The proof of this theorem consists of two parts: An abstract theorem based on a result of \cite{BorichevEtAl2009} in complex analysis and  uniform resolvent estimates. The latter will be needed to handle the potentials appearing in the theorem. If the potential $V$ satisfies much stronger assumption (for example, $V$ compactly supported), then the uniform resolvent estimates are much easier. To illustrate this, we give a proof of Theorem \ref{Main theorem} for the following special case: $H_0=(-\Delta)^{s/2}$ with $d/2<s<d$, and $V\in L^{d/s}_{\rm comp}(X)$.

\begin{proof}[Proof of Theorem \ref{Main theorem} (special case)]
We claim that
\begin{align}\label{kernel bound near zero}
|R_0(x-y;z)|\leq C(r)|x-y|^{s-d},\quad |x-y|\leq r,
\end{align}holds for a constant independent of $z\in\rho(H_0)$.
By homogeneity, it suffices to prove \eqref{kernel bound near zero} for $|z|=1$. Let $\chi$ be a radial bump function supported in $\{1/2\leq |\xi|\leq 3/2\}$ and equal to $1$ in $\{3/4\leq |\xi|\leq 5/4\}$. Then the estimate \eqref{kernel bound near zero} holds for $(1-\chi^2(D))R_0(z)$ by standard estimates, see e.g.\ \cite[Prop. VI.4.1]{Stein1993}. For $\chi^2(D)R_0(z)$ we get an $\mathcal{O}(1)$ kernel bound by comparison with the Hilbert transform (see \eqref{Fourier transform of Hilbert transform}).
Assume that $\supp(V)\subset B(0,R/2)$. Then it follows from \eqref{kernel bound near zero} and the Hardy-Littlewood-Sobolev inequality (recall $d/2<s<d$) that
\begin{align}\label{Hilbert Schmidt bound compactly supported V}
\||V|^{1/2}R_0(z)V^{1/2}\|_{\mathfrak{S}^2}^2\leq C(r)^2\int\int\frac{|V(x)||V(y)|}{|x-y|^{2(s-d)}}\rd x\rd y\leq C\|V\|_{L^{d/s}}^2.
\end{align}
The second claim of the theorem follows by a straightforward application of the Birman-Schwinger principle: Suppose $z\in\rho(H_0)$ is an eigenvalue of $H_0+V$. Then there exists $f\in L^2(X)$ such that $-Vf=(H_0-z)f$. Applying $R_0(z)$ to both sides, multiplying by $|V|^{1/2}$, and then setting $|V|^{1/2}f=g$, we obtain 
\begin{align*}
-g=|V|^{1/2}R_0(z)V^{1/2}g,
\end{align*}
where $V^{1/2}=|V|^{1/2}\sgn(V)$ (here, $\sgn(V)$ is the complex signum function).
Clearly, $g\in L^2(X)$ again, and hence the operator $|V|^{1/2}R_0(z)V^{1/2}$ has an eigenvalue $-1$. But this means that its norm is at least $1$. Since the norm of an operator is bounded from above by its Hilbert-Schmidt norm, \eqref{Hilbert Schmidt bound compactly supported V} tells us that $1\leq C\|V\|_{L^{d/s}}^2$. Therefore, a necessary condition for the existence of an eigenvalue in $\rho(H_0)$ is that $\|V\|_{L^{d/s}}$ is bigger than some positive constant.\footnote{The case where the eigenvalue can be embedded in $\sigma(H_0)$ follows from \cite[Proposition 3.1]{FrankSimon2015}.} We turn to the proof of \eqref{sum of eigenvalues converges}. By the previous argument (Birman-Schwinger principle) and the theory of infinite determinants (see e.g.\ \cite[Thm. 9.2]{Simon2005}) it follows that the eigenvalues of $H_0+V$ in $\rho(H_0)$ coincide with the zeros of the (regularized) determinant
\begin{align*}
h(z):=\Det_2(I+|V|^{1/2}R_0(z)V^{1/2}),\quad z\in\rho(H_0).
\end{align*}
This is an analytic function, and there is the well-known bound 
\begin{align}\label{bound for determinant}
\log|h(z)|\leq \frac{1}{2}\||V|^{1/2}R_0(z)V^{1/2}\|_{\mathfrak{S}^2}^2.
\end{align}
From \eqref{Hilbert Schmidt bound compactly supported V} and \eqref{bound for determinant} it thus follows that $h$ is bounded. As noted in \cite{FrankSabin14}, one can apply a Jensen-type inequality for the upper half-plane $\C^+$ to the map $w\mapsto h(w^2)$, 
%see e.g. (2.3) in Garnett!!!,
giving
\begin{align*}
\sum_{h(z)=0}\frac{\im\sqrt{z}}{1+|z|}<\infty.
\end{align*}
When we restrict $z$ to a compact subset $K$ of $\C\setminus\{0\}$, this sum is comparable to the one in \eqref{sum of eigenvalues converges}.
\end{proof}

\section{Uniform $L^p\to L^{p'}$ resolvent estimates}

\begin{theorem}\label{theorem uniform Sobolev inequality fractional Laplacian}
Let $H_0=T(D)$ where $T(D)$ is one of the operators in Assumption~\ref{unperturbed}, and let $K$ be a compact subset of $\C\setminus \Lambda_c(H_0)$. Then the following estimates hold for all $z\in K\cap\rho(H_0)$.
\begin{itemize}
\item[a)] If $d>s\geq 2d/(d+1)$ and $2d/(d+s)\leq p\leq 2(d+1)/(d+3)$, then for all $f\in C_0^{\infty}(X)$, we have that
\begin{align}\label{uniform Sobolev inequality fractional Laplacian big s}
\|R_0(z)f\|_{L^{p'}}\leq C(K)\|f\|_{L^p}.
\end{align}
\item[b)] If $s<2d/(d+1)$, then for all $f\in C_0^{\infty}(X)$, we have that
\begin{align}\label{uniform Sobolev inequality fractional Laplacian small s}
\|R_0(z)f\|_{L^{\frac{2d}{d-s}}+L^{\frac{2(d+1)}{d-1}}}\leq C(K)\|f\|_{L^{\frac{2d}{d+s}}\cap L^{\frac{2(d+1)}{d+3}}}.
\end{align}
\end{itemize}
\end{theorem}

\begin{corollary}\label{corollary uniform Sobolev inequality fractional Laplacian}
Let $V$ be a complex-valued potential satisfying Assumption \ref{potential}, and let $K$ be a compact subset of $\C\setminus \Lambda_c(H_0)$. Then the following estimates hold for all $z\in K\cap\rho(H_0)$.
\begin{align}\label{Birman-Schwinger bound L2L2}
\left\||V|^{1/2}R_0(z)|V|^{1/2}\right\|_{L^2\to L^2}\leq C(K,\|V\|).
\end{align}
\end{corollary}

\begin{proof}
We prove the case where $V$ satisfies Assumption \ref{potential} b). The proof of case~a) is similar and easier still. Denote $X=L^{\frac{2d}{d+s}}\cap L^{\frac{2(d+1)}{d+3}}$ and $X^*=L^{\frac{2d}{d-s}}+L^{\frac{2(d+1)}{d-1}}$ its dual.
Theorem \ref{theorem uniform Sobolev inequality fractional Laplacian} yields that for $f,g\in C_0^{\infty}(X)$,
\begin{align*}
|\langle R_0(z)|V|^{1/2}f,|V|^{1/2}g\rangle|&\leq \|R_0(z)|V|^{1/2}f\|_{X^*}\||V|^{1/2}g\|_{X}\\
&\leq C(K)\||V|^{1/2}f\|_{X}\||V|^{1/2}g\|_{X}\\
&\leq C(K)\||V|^{1/2}\|^2_{L^2\to X}\|f\|_{L^2}\|g\|_{L^2}.
\end{align*}
Taking the supremum over normalized functions $f,g\in L^2$, and using H\"older's inequality, we get
\begin{align*}
\left\||V|^{1/2}R_0(z)|V|^{1/2}\right\|_{L^2\to L^2}\leq C(K)\max\{\|V\|_{L^{d/s}},\|V\|_{L^{(d+1)/2}}\}.
\end{align*} 
\end{proof}

\begin{lemma}\label{lemma resolvent estimate low frequency}
Let $T\in C^2(\R^d,\R)$ and let $K$ be a compact subset of $\C\setminus\Lambda_c(T)$. Assume that the Gaussian curvature of the level sets $S_{\lambda}=\set{\xi\in\R^d}{T(\xi)=\lambda}$ never vanishes for $\lambda\in K\cap\R$.
If $\chi$ is a smooth compactly supported function, then for $1\leq p\leq 2(d+1)/(d+3)$ and for all $z\in K\cap\rho(H_0)$, we have the estimate
\begin{align}\label{inequality lemma low frequency}
\|\chi^2(D)R_0(z)\|_{L^p\to L^{p'}}\leq C(K,T,\chi).
\end{align}
%The constant depends on $d,p,K,T,\chi$, but not on $z$.
\end{lemma}

\begin{proof}
We write $z=\lambda+\I\epsilon$. By a limiting argument, we may assume $\epsilon\neq 0$ for the rest of the proof.
Moreover, we may suppose without loss of generality that $\chi$ is supported in a small neighborhood of the origin, $T(0)=0$ and $\nabla_{\xi}T(0)=\rho e_1$ for some $\rho>0$; here $e_1$ is the unit vector in the $\xi_1$-direction. Otherwise, we use a partition of unity and a linear change of coordinates.
By the implicit function theorem, for $\xi$ in a neighborhood of $\supp\chi$ and $\lambda$ in a neighborhood of $0$, we have
\begin{align}\label{implicit function theorem}
T(\xi)-\lambda=e(\xi,\lambda)(\xi_1-h(\xi',\lambda))
\end{align}
where $\xi=(\xi_1,\xi')\in\R\times\R^{d-1}$, and $h$, $e$ are real-valued smooth functions, $e$ being bounded away from zero. Locally, $e(\xi,\lambda)$ is given by the expression
\begin{align}\label{exilambda}
e(\xi,\lambda)=\left(\int_0^1T_{\xi_1}(t\xi_1+(1-t)h(\xi',\lambda),\xi')\rd t\right).
\end{align}
We extend $e(\xi,\lambda)$ and $h(\xi,\lambda)$ to real-valued functions in $C_b^{\infty}(\R^d)$ with $C^k$-seminorms independent of $\lambda\in\R$ and such that $e_2\geq e(\xi,\lambda)\geq e_1>0$. Let us set $b(\xi,\lambda):=e(\xi,\lambda)^{-1}$. 
Since, by Young's inequality, $(b\chi)(D,\lambda)$ is bounded on $L^p$ for all $p\in[1,\infty]$, it is sufficient to prove
\begin{align}\label{inequality pricipal type}
\sup_{z\in K\cap\rho(H_0)}\|\chi(D)(D_1-h(D',\lambda)-\I\epsilon b(D,\lambda))^{-1}\|_{L^p\to L^{p'}}<\infty.
\end{align}
Denote by $K_{\lambda,\epsilon}(\cdot)$ the inverse Fourier transform of $\xi\mapsto \chi(\xi)(\xi_1-h(\xi',\lambda)-\I \epsilon b(\xi,z))^{-1}$. Then
\begin{align*}
\chi(D)(D_1-h(D',\lambda)-\I\epsilon b(D,\lambda))^{-1}g(x)=\int_{-\infty}^{\infty}K_{\lambda,\epsilon}(x_1-y_1,\cdot)*g(\cdot,y_1)(x')\rd y_1
\end{align*}
We claim that
\begin{align}
\|K_{\lambda,\epsilon}(x_1-y_1,\cdot)*g\|_{L^{\infty}(\R^{d-1})}&\leq C(1+|x_1-y_1|)^{-\frac{d-1}{2}}\|g\|_{L^1(\R^{d-1})},\label{L1-Linfty}\\
\|K_{\lambda,\epsilon}(x_1-y_1,\cdot)*g\|_{L^{2}(\R^{d-1})}&\leq C\|g\|_{L^2(\R^{d-1})}\label{L2-L2}
\end{align}
with $C$ independent of $\lambda$, $\epsilon$.
Interpolating between \eqref{L1-Linfty} and \eqref{L2-L2}, we get
\begin{align*}
\|K_{\lambda,\epsilon}(x_1-y_1,\cdot)*g\|_{L^{p^{'}}(\R^{d-1})}&\leq C(1+|x_1-y_1|)^{-(d-1)\left(\frac{1}{p}-\frac{1}{2}\right)}\|g\|_{L^p(\R^{d-1})}.
\end{align*}
A standard argument using Minkowski's inequality for integrals and fractional integration in one dimension then yields for $p=2(d+1)/(d+3)$,
\begin{align*}
&\|\chi(D)(D_1-h(D',\lambda)-\I \epsilon b(D,\lambda))^{-1}g\|_{L^{p'}(\R^d)}\\
&\leq \left\|\int_{\R}\left\|K_{\lambda,\epsilon}(x_1-y_1,\cdot)*g(y_1,\cdot)\right\|_{L_{x'}^{p'}(\R^{d-1})}\rd y_1\right\|_{L_{x_1}^{p'}(\R)}\\
&\leq C\left\|\int_{\R}|x_1-y_1|^{-(d-1)\left(\frac{1}{p}-\frac{1}{2}\right)} \|g(y_1,\cdot)\|_{L^{p}_{x'}(\R^{d-1})}\rd y_1\right\|_{L_{x_1}^{p'}(\R)}\\
&\leq C \left\|\|g(x_1,\cdot)\|_{L^p_{x'}(\R^{d-1})}\right\|_{L_{x_1}^{p}(\R)}
=C\|g\|_{L^p(\R^d)}.
\end{align*}
This proves \eqref{inequality pricipal type} for $p=2(d+1)/(d+3)$. Interpolating with the $L^1\to L^{\infty}$-bound (which follows from \eqref{Fourier transform of Hilbert transform} below), the claim \eqref{inequality pricipal type} follows.

It remains to prove inequalities \eqref{L1-Linfty}--\eqref{L2-L2}. The convolution kernel $K_{\lambda,\epsilon}$ is given by
\begin{align*}
K_{\lambda,\epsilon}(x_1,x')
=\int_{\R^{d-1}}\e^{2\pi\I[x'\cdot\xi'+x_1h(\xi',\lambda)]}\psi_{\lambda,\epsilon}(\xi',x_1)\rd\xi',
\end{align*}
where
\begin{align*}
\psi_{\lambda,\epsilon}(\xi',x_1)=\int_\R \e^{2\pi\I x_1\xi_1}\chi(\xi_1,\xi')(\xi_1-\I \epsilon b(\xi_1,\xi',\lambda))^{-1}\rd\xi_1
\end{align*}
We claim that for all $\alpha\in\N^d_0$, there are constants $C_{\alpha}>0$ such that
\begin{align}\label{symbol bounds psi}
\sup_{\lambda+\I \epsilon\in K}\sup_{\xi'\in\R^{d-1}}\sup_{x_1\in\R}|\partial_{\xi'}^{\alpha}\psi_{\lambda,\epsilon}(\xi',x_1)|\leq C_{\alpha}.
\end{align}
%with constants $C_{\alpha}$ independent of $x_1\in\R$ and $z\in K\cap\rho(H_0)$.

Inequality \eqref{L2-L2} then follows from Plancherel's theorem, and inequality \eqref{L1-Linfty} follows from stationary phase estimates, see e.g.\ \cite[Prop. VIII.2.6]{Stein1993}. Here, we are using that
\begin{align*}
h(0,\lambda)=0,\quad\nabla_{\xi'}h(0,\lambda)=0,\quad D_{\xi'}^2h(0,\lambda)\quad\mbox{is nondegenerate}.
\end{align*}
The last assertion follows from the assumption that the Gaussian curvature of the level sets $S_{\lambda}$, given by
\begin{align*}
(1+|\nabla_{\xi'}h(\cdot,\lambda)|^2)^{-(d+1)/2}\Det(D_{\xi'}^2h(\cdot,\lambda))
\end{align*}
in local coordinates, never vanishes. By compactness of $K\subset\C\setminus\Lambda_c(T)$, there is a uniform bound from below for all $\lambda\in K\cap\R$.

To prove \eqref{symbol bounds psi}, we note that $\partial_{\xi'}^{\alpha} \chi(\cdot,\xi')\in\mathcal{S}(\R)$ with Schwartz norms bounded independently of $\xi'\in\R^{n-1}$. Assume that $\alpha=0$. By the convolution theorem, we have 
\begin{align*}
|\psi_{\lambda,\epsilon}(\xi',x_1,\epsilon)|\leq \|\mathcal{F}_1^{-1}\chi(\cdot,\xi')\|_{L^1(\R)}\|\mathcal{F}_1^{-1}(\cdot-\I\epsilon b(\cdot,\xi',\lambda))^{-1}\|_{L^{\infty}(\supp(\pi_1\chi))},
\end{align*}
where $\pi_1$ denotes the projection onto the first coordinate. The second (inverse) Fourier transform has to be understood in the sense of tempered distributions. By a change of variables $\xi_1 \to \epsilon\xi_1$, it is sufficient to prove that
\begin{align}\label{Fourier transform of Hilbert transform}
\sup_{\lambda+\I \epsilon\in K}\sup_{\xi'\in\R^{d-1}}\left\|\int_{\R}\e^{2\pi \I x_1\xi_1}(\xi_1-\I b(\epsilon\xi_1,\xi',\lambda))^{-1}\rd\xi_1\right\|_{L^{\infty}_{x_1}(\R)}\leq C.
\end{align}
If we set $u_{\lambda,\epsilon}(\xi_1,\xi'):=(\xi_1-\I b(\epsilon\xi_1,\xi',\lambda))^{-1}$, then \eqref{Fourier transform of Hilbert transform} means that for any $\varphi\in\mathcal{S}(\R)$, we have
\begin{align}\label{Fourier transform of Hilbert transform distributional sense}
\sup_{\lambda+\I \epsilon\in K}\sup_{\xi'\in\R^{d-1}}|\langle u_{\lambda,\epsilon}(\cdot,\xi'),\widehat{\varphi}\rangle|\leq C\|\varphi\|_{L^1(\R)}.
\end{align}
Here, $\langle\cdot,\cdot\rangle$ denotes the duality bracket between $\mathcal{S}'(\R)$ and $\mathcal{S}(\R)$.
Let $\chi_{\leq }\in\mathcal{S}(\R)$ be such that $\widehat{\chi}_{\leq}\in C_c^{\infty}(\R,[0,1])$, $\supp \widehat{\chi}_{\leq}\subset [-2,2]$ and $\widehat{\chi}_{\leq}\equiv 1$ in $[-1,1]$. We set $u_{\leq,\lambda,\epsilon}:=u_{\lambda,\epsilon}\widehat{\chi}_{\leq}\in\mathcal{S}'(\R)$ and $u_{\geq,\lambda,\epsilon}:=u_{\lambda,\epsilon}(1-\widehat{\chi}_{\leq})\in\mathcal{S}'(\R)$. 
Then
\begin{align}\label{first part of proof of Fourier transform of Hilbert transform distributional sense}
|\langle u_{\leq,\lambda,\epsilon},\widehat{\varphi}\rangle|
\leq \|\widehat{\chi}_{\leq}\|_{L^1(\R)}\|(\cdot-\I b(\cdot,\xi',\lambda))^{-1}\|_{L^{\infty}(\R)}\|\widehat{\varphi}\|_{L^{\infty}(\R)}\leq C\|\varphi\|_{L^1(\R)},
\end{align}
with a constant $C$ independent of $\xi'$, $\lambda$ and $\epsilon$. Further, we write
\begin{align*}
\langle u_{\geq,\lambda,\epsilon},\widehat{\varphi}\rangle 
=\langle p.v.\frac{1}{\xi_1},(1-\widehat{\chi}_{\leq})\widehat{\varphi}\rangle
+\langle  u_{\geq,\lambda,\epsilon}-p.v.\frac{1}{\xi_1},(1-\widehat{\chi}_{\leq})\widehat{\varphi}\rangle.
\end{align*}
Since the Hilbert transform $\xi_1\mapsto p.v.1/\xi_1$ has Fourier transform equal to $-\I\pi\sgn(\cdot)$, it follows that
\begin{align*}
|\langle p.v.\frac{1}{\xi_1},(1-\widehat{\chi}_{\leq})\widehat{\varphi}\rangle|\leq \pi(1+\|\chi_{\leq}\|_{L^1(\R)})\|\varphi\|_{L^1(\R)}.
\end{align*}
On the other hand, 
\begin{align*}
|\langle u_{\geq,\lambda,\epsilon}-p.v.\frac{1}{\xi_1},(1-\widehat{\chi}_{\leq})\widehat{\varphi}\rangle|
&\leq \|\widehat{\varphi}\|_{L^{\infty}(\R)}\int_{|\xi_1|\geq 1}\frac{ |b(\epsilon\xi_1,\xi',\lambda)|}{|\xi_1-\I b(\epsilon\xi_1,\xi',\lambda)||\xi_1|}\rd\xi_1\\
&\leq C\|\varphi\|_{L^1(\R)},
\end{align*}
where $C$ is independent of $\xi'$, $\lambda$ and $\epsilon$. This completes the proof of \eqref{Fourier transform of Hilbert transform distributional sense} and hence of \eqref{symbol bounds psi} for $\alpha=0$. 

For $|\alpha|\geq 1$, a dominated convergence argument implies that  
\begin{align*}
\partial_{\xi'}^{\alpha}\psi_{\lambda,\epsilon}(\xi',x_1)=\int_\R \e^{2\pi\I x_1\xi_1}\partial_{\xi'}^{\alpha}\left(\chi(\xi_1,\xi')(\xi_1-\I\epsilon b(\xi_1,\xi',\lambda))^{-1}\right)\rd\xi_1
\end{align*}
By Leibniz' formula,
\begin{align*}
\partial_{\xi'}^{\alpha}\left(\chi(\xi_1,\xi')(\xi_1-\I\epsilon b(\xi_1,\xi',\lambda))^{-1}\right)
=\sum_{\beta\leq \alpha}\begin{pmatrix}
\alpha\\\beta
\end{pmatrix}
\partial_{\xi'}^{\alpha-\beta}\chi(\xi_1,\xi')\partial_{\xi'}^{\beta}(\xi_1-\I\epsilon b(\xi_1,\xi',\lambda))^{-1},
\end{align*}
and by Faà di Bruno's formula,
\begin{align*}
\partial_{\xi'}^{\beta}(\xi_1-\I\epsilon b(\xi_1,\xi',\lambda))^{-1}=\sum_{\beta_1+\ldots+\beta_l=\beta}C_{\beta_1,\ldots,\beta_k}(-\I\epsilon)^l\frac{\partial_{\xi'}^{\beta_1} b(\xi_1,\xi',\lambda)\ldots \partial_{\xi'}^{\beta_l} b(\xi_1,\xi',\lambda)}{(\xi_1-\I\epsilon b(\xi_1,\xi',\lambda))^{l+1}}
\end{align*} 
for some constants $C_{\beta_1,\ldots,\beta_k}$, where the sum is taken over all partitions of $\beta$. 
Since $b(\cdot,\xi',\lambda)\in C_b^{\infty}(\R)$ with $C^k$-seminorms uniform in $\xi'$ and $\lambda$, it remains to prove that for any $\widetilde{\chi}\in C_c^{\infty}(\R)$ and $l\geq 1$, we have
\begin{align*}
\left\|\epsilon^l\int_{\R} \e^{2\pi\I x_1\xi_1}\frac{\widetilde{\chi}(\xi_1)}{(\xi_1-\I\epsilon b(\xi_1,\xi',\lambda))^{1+l}}\rd\xi_1\right\|_{L^{\infty}_{x_1}(\R)}\leq C,
\end{align*}
with $C$ independent of $\xi'$, $\lambda$ and $\epsilon$. Since $| b(\xi_1,\xi',\lambda)|\geq e_2^{-1}>0$, we see by a change of variables $\xi_1\to \epsilon\xi_1$ that the integral is absolutely convergent, with a uniform bound in $\xi'$, $\lambda$ and $\epsilon$.
\end{proof}

\begin{proof}[Proof of Theorem \ref{theorem uniform Sobolev inequality fractional Laplacian}]
We prove case b), case a) being similar and easier. Consider $T(\xi)=|\xi|^s$ or $T(\xi)=(1+|\xi|^2)^{s/2}-1$ first. These are Fourier multipliers corresponding to radial functions. Thus the level sets $S_{\lambda}$ are spheres of size $\mathcal{O}(1)$ for all $\lambda\in K$. In particular, their Gaussian curvature is non-vanishing. Let $\Omega,\Omega'\subset\C$ be open sets such that $K\subset\subset \Omega \subset\subset \Omega'$. Pick a bump function $\chi\in\C^{\infty}_c(\R^d;[0,1])$ such that $\supp(\chi)\subset T^{-1}(\Omega'\cap\R)$ (note that $T$ is proper, i.e.\ the preimage of a compact set under $T$ is compact) and $\chi=1$ on $T^{-1}(\Omega\cap\R)$. 
Lemma~\ref{lemma resolvent estimate low frequency} implies
\begin{align}\label{low frequency bound case b in the proof of LpLp'}
\|\chi^2(D)R_0(z)\|_{L^{\frac{2(d+1)}{d+3}}\to L^{\frac{2(d+1)}{d-1}}}\leq C(K,\chi).
\end{align}
By the choice of $\chi$, we have
\begin{align*}
\int_{\R^d}\frac{|(1-\chi(\xi))|^2}{|T(\xi)-z|^2}(1+|\xi|^2)^{s/2}|\widehat{f}(\xi)|^2\rd\xi\leq \frac{1}{\dist(\partial\Omega,K)^2}\int_{\R^d}(1+|\xi|^2)^{-s/2}|\widehat{f}(\xi)|^2\rd\xi.
\end{align*}
Then
\begin{align}\label{elliptic Hs bound}
\|(1-\chi^2(D))R_0(z)f\|_{H^{s/2}}\leq \frac{1}{\dist(\partial\Omega,K)}\|f\|_{H^{-s/2}}.
\end{align}
Together with the (dual) Sobolev embedding
\begin{align*}
H^{s/2}(\R^d)\hookrightarrow L^{\frac{2d}{d-s}}(\R^d),\quad L^{\frac{2d}{d+s}}(\R^d)\hookrightarrow H^{-s/2}(\R^d),
\end{align*}
the estimate \eqref{elliptic Hs bound} implies that
\begin{align}\label{high frequency bound case b in the proof of LpLp'}
\|(1-\chi^2(D))R_0(z)\|_{L^{\frac{2d}{d+s}}\to L^{\frac{2d}{d-s}}}\leq C(K,\chi).
\end{align}
Let $f=f_1+f_2$ where $f_1=\chi^2(D)f$. Then, by \eqref{low frequency bound case b in the proof of LpLp'} and \eqref{high frequency bound case b in the proof of LpLp'},
\begin{align*}
\|R_0(z)f_1\|_{L^{\frac{2d}{d-s}}}+\|R_0(z)f_2\|_{L^{\frac{2(d+1)}{d-1}}}\leq C(K,\chi)\|f\|_{L^{\frac{2d}{d+s}}\cap L^{\frac{2(d+1)}{d+3}}}.
\end{align*}
This proves the claim in this case.

Now consider $H_0\in\{\mathcal{D}_0,\mathcal{D}_1\}$. Since we have $\mathcal{D}_0^2=-\Delta$ and $\mathcal{D}_1^2=1-\Delta$, the resolvents are given by
\begin{align}
(\mathcal{D}_0-z)^{-1}&=(\mathcal{D}_0+z)(-\Delta-z^2)^{-1}\label{Dirac resolvent},\\
(\mathcal{D}_1-z)^{-1}&=(\mathcal{D}_1+z)(-\Delta-(z^2-1))^{-1}.
\end{align}
The proof of the claim is then a straightforward adaptation of the previous argument. Note that in the analogue of \eqref{low frequency bound case b in the proof of LpLp'} in the case of $\mathcal{D}_0$,
\begin{align*}
\|(\mathcal{D}_0+z)\chi^2(D)(-\Delta-z^2)^{-1}\|_{L^{\frac{2(d+1)}{d+3}}\to L^{\frac{2(d+1)}{d-1}}}\leq C(K,\chi),
\end{align*}
the multiplier $\mathcal{D}_0+z$ can be absorbed into $\chi^2(D)$ since the frequencies are bounded. 
\end{proof}

\section{Uniform resolvent estimates in Schatten spaces}

% Old version. Do not delete!
%\begin{theorem}\label{main theorem schatten norms}
%Let $H_0=T(D)$, and let $V$ be a complex-valued potential such that Assumptions \ref{unperturbed}--\ref{potential} hold. Then there exists a positive function $N(z)$, defined on $\rho(H_0)$ and having a continuous extension up to $\C\setminus\Lambda_c(H_0)$, such that the following inequalities hold.
%\begin{itemize}
%\item[a)] If $s\geq 2d/(d+1)$ and $V\in L^q$ with $2d/(d+1)< q\leq (d+1)/2$, then
%\begin{align}\label{eq. Schatten bound theorem a}
%\||V|^{1/2}R_0(z)|V|^{1/2}\|_{\mathfrak{S}^{q(d-1)/(d-q)}}\leq N(z)\|V\|_{L^q}
%\end{align}
%\item[a')] If $s\geq 2d/(d+1)$ and $V\in L^q$ with $d/s\leq q\leq 2d/(d+1)$ and $\epsilon>0$, then there exists $C_{\epsilon}>0$ such that
%\begin{align}\label{eq. Schatten bound theorem a'}
%\||V|^{1/2}R_0(z)|V|^{1/2}\|_{\mathfrak{S}^{q(d-1)/(d-q)+\epsilon}}\leq C_{\epsilon}N(z)\|V\|_{L^q}
%\end{align}
%\item[b)] If $s< 2d/(d+1)$ and $V\in L^{d/s}\cap L^{(d+1)/2}$ and $\epsilon>0$, then there exists $C_{\epsilon}>0$ such that
%\begin{align}\label{eq. Schatten bound theorem b}
%\||V|^{1/2}R_0(z)|V|^{1/2}\|_{\mathfrak{S}^{\max\{d+1,d/s+\epsilon\}}}\leq C_{\epsilon}N(z)\|V\|_{L^{d}\cap L^{(d+1)/2}}.
%\end{align}
%\end{itemize}
%\end{theorem}

\begin{theorem}\label{main theorem schatten norms}
Let $H_0=T(D)$, and let $V$ be a complex-valued potential such that Assumptions \ref{unperturbed}--\ref{potential} hold. Then there exists a positive function $N(z)$, defined on $\rho(H_0)$ and having a continuous extension up to $\C\setminus\Lambda_c(H_0)$, such that the following inequalities hold.
\begin{itemize}
\item[a)] If $2d/(d+1)\leq s\leq (d+1)/2$ and $d/s\leq q\leq (d+1)/2$ or if $(d+1)/2<s<d$ and $2d/(d+1)\leq q\leq (d+1)/2$, then there exists $C_{q,d,s}>0$ such that
\begin{align}\label{eq. Schatten bound theorem a}
\||V|^{1/2}R_0(z)|V|^{1/2}\|_{\mathfrak{S}^{q(d-1)/(d-q)}}\leq C_{q,d,s} N(z)\|V\|_{L^q}.
\end{align}
\item[a')] If $(d+1)/2<s<d$ and $d/s\leq q<2d/(d+1)$, then, for every $\epsilon>0$, there exists $C_{q,d,s,\epsilon}>0$ such that
\begin{align}\label{eq. Schatten bound theorem a'}
\||V|^{1/2}R_0(z)|V|^{1/2}\|_{\mathfrak{S}^{q(d-1)/(d-q)+\epsilon}}\leq C_{q,d,s,\epsilon}N(z)\|V\|_{L^q}.
\end{align}
\item[b)] If $0<s< 2d/(d+1)$, then, for every $\epsilon>0$, there exists $C_{d,s,\epsilon}>0$ such that
\begin{align}\label{eq. Schatten bound theorem b}
\||V|^{1/2}R_0(z)|V|^{1/2}\|_{\mathfrak{S}^{\max\{d+1,d/s+\epsilon\}}}\leq C_{d,s,\epsilon}N(z)\|V\|_{L^{d/s}\cap L^{(d+1)/2}}.
\end{align}
\end{itemize}
\end{theorem}

\begin{remark}\label{remark main theorem schatten norms}
From scaling arguments, one can obtain the following expressions for the function $N(z)$. A proof is given following the proof of Theorem \ref{main theorem schatten norms}.
\begin{itemize}
\item If $H_0=(-\Delta)^{s/2}$, then we may choose $N(z):=|z|^{\frac{d}{sq}-1}$ if $s\geq 2d/(d+1)$ and $N(z):=(1+|z|)^{\frac{2d}{s(d+1)}-1}$ if $s< 2d/(d+1)$.
\item If $H_0=(1-\Delta)^{s/2}-1$, then 
\begin{align*}
N(z):=
\begin{cases}
|z|^{\frac{d}{2q}-1}\quad&\mbox{if }|z|<1,\\
|z|^{\frac{d}{sq}-1}\quad&\mbox{if }|z|\geq 1,
\end{cases}
\quad \mbox{and  } s\geq  2d/(d+1),
\end{align*}
and 
\begin{align*}
N(z):=
\begin{cases}
|z|^{\frac{s}{2}-1}\quad&\mbox{if }|z|<1,\\
|z|^{\frac{2d}{s(d+1)}-1}\quad&\mbox{if }|z|\geq 1,
\end{cases}
\quad \mbox{and  } s<  2d/(d+1).
\end{align*}
\item If $H_0=\mathcal{D}_0$, then $N(z):=(1+|z|)^{\frac{d-1}{d+1}}$.
\item If $H_0=\mathcal{D}_1$, then 
\begin{align*}
N(z):=
\begin{cases}
|z^2-1|^{-\frac{1}{2}}\quad&\mbox{if }|z^2-1|<1,\\
|z|^{\frac{d-1}{d+1}}\quad&\mbox{if }|z^2-1|\geq 1.
\end{cases}
\end{align*}
\end{itemize}
\end{remark}

\begin{remark}\label{rem. question remark 8 in FS}
As will be evident from the proof, Theorem \ref{main theorem schatten norms} remains valid for $s=d$ under the additional restriction $q>d/s$. In particular, the following extension of Theorem \ref{main theorem schatten norms} (a') to the case $d=2$, $s=2$ holds. For $1<q\leq 4/3$ and for $\epsilon>0$ there is a constant $C_{q,\epsilon}$ such that
\begin{align}\label{eq. 2d Laplacian}
\||V|^{1/2}(-\Delta-z)^{-1}|V|^{1/2}\|_{\mathfrak{S}^{q/(2-q)+\epsilon}}\leq C_{q,\epsilon}|z|^{\frac{1}{q}-1}\|V\|_{L^q}.
\end{align}
This settles a question posed in \cite[Remark 8]{FrankSabin14} and improves the result of \cite[Proposition 9]{FrankSabin14} where \eqref{eq. 2d Laplacian} was proved for the Hilbert-Schmidt norm instead.
\end{remark}

\begin{lemma}\label{lemma schatten estimate low frequency}
Let $T\in C^2(\R^d,\R)$ and let $\chi$ be a smooth compactly supported function such that $T$ has no critical points in $\supp(\chi)$. Assume that the Gaussian curvature of the level sets $S_{\lambda}=\set{\xi\in\supp(\chi)}{T(\xi)=\lambda}$ never vanishes. Then for $1\leq q\leq (d+1)/2$ and for all $z\in\rho(T)$, we have the estimate
\begin{align}\label{inequality Schatten low frequency}
\||V|^{1/2}\chi^2(D)R_0(z)|V|^{1/2}\|_{\mathfrak{S}^{\alpha_q}}\leq C(T,\chi)\|V\|_{L^{q}},
\end{align}
%The constant $C$ depends on $d,p,T,\chi$, but not on $z$.
where $\alpha_q=q(d-1)/(d-q)$ for $2d/(d+1)\leq q\leq (d+1)/2$ and $\alpha_q=q(d-1)/(d-q)+\epsilon$ for $1\leq q< 2d/(d+1)$ and arbitrary $\epsilon>0$; in the latter case, the constant $C(T,\chi)$ also depends on $\epsilon$.
\end{lemma}

\begin{proof}
1. Let $\Omega,\Omega'\subset\R$ be such that $T(\supp(\chi))\subset\subset \Omega\subset\subset \Omega'$. Assume first that $\lambda=\re z\in \C\setminus \Omega'$. 
%Let $\widetilde{\chi}$ be another bump function with the same properties as $\chi$ and such that $\widetilde{\chi}=1$ on the support of $\chi$. 
Since $q(d-1)/(d-q)\geq q$, it follows from~\eqref{embeddings Schatten spaces} and from H\"older's inequality in Schatten spaces \cite[Thm. 2.8]{Simon2005} that
\begin{equation}\label{prelude to Kato-Seiler-Simon}
\||V|^{1/2}\chi^2(D)R_0(z)|V|^{1/2}\|_{\mathfrak{S}^{q(d-1)/(d-q)}}\leq 
\||V|^{1/2}\chi^2(D)|R_0(z)|^{1/2}\|_{\mathfrak{S}^{2q}}^2.
\end{equation}
By the Kato-Seiler-Simon inequality (see e.g.\ \cite[Thm. 4.1]{Simon2005}), we have the estimate
\begin{align*}
\||V|^{1/2}\chi^2(D)|R_0(z)|^{1/2}\|_{\mathfrak{S}^{2q}}\leq C\|V\|_{L^q}\|\chi(\cdot)(T(\cdot)-z)^{-1}\|_{L^{2q}}
\leq \frac{C\|\chi\|_{L^{2q}}}{\dist(\partial\Omega',\Omega)}\|V\|_{L^q},
\end{align*}
Together with \eqref{prelude to Kato-Seiler-Simon}, the previous estimate yields
\begin{align}\label{Kato-Seiler-Simon}
\||V|^{1/2}\chi^2(D)R_0(z)|V|^{1/2}\|_{\mathfrak{S}^{q(d-1)/(d-q)}}\leq C(T,\chi)\|V\|_{L^q}.
\end{align}

2. Assume now that $\lambda\in \Omega'$. We first prove \eqref{inequality Schatten low frequency} for $z=\lambda\pm \I 0$ and for $2d/(d+1)< q\leq (d+1)/2$. Note that the limits
\begin{align}
(T(\xi)-\lambda\pm\I 0)^{-1}&=p.v.(T(\xi)-\lambda)^{-1}\mp\I\pi\delta(T(\xi)-\lambda),\label{distributional limit T-lambda}
\\
(T(\xi)-\lambda\pm\I 0)^{-\sigma}&=(T(\xi)-\lambda)_+^{-\sigma}+\e^{-\I\pi\sigma}(T(\xi)-\lambda)_-^{-\sigma},
\end{align}
where, $\sigma\in \C\setminus\{\ldots,-2,-1,0,1\}$, exist in the sense of tempered distributions, see e.g.\ \cite{GelfandShilov1977}.
We follow the outline of the proof of \cite[Theorem 12]{FrankSabin14} and apply complex interpolation to the family $\zeta\mapsto|V|^{\zeta/2}\chi^2(D)R_0(\lambda\pm\I 0)^{\zeta}|V|^{\zeta/2}$. We shall prove the following bounds: For $\re\,\zeta=0$,
\begin{align}\label{complex interpolation easy}
\||V|^{\zeta/2}\chi^2(D)R_0(\lambda\pm\I 0)^{\zeta}V^{\zeta/2}\|_{\mathfrak{S}^{\infty}}&\leq C(T,\chi,\im\,\zeta),
\end{align}
and for $1\leq \re\,\zeta\leq (d+1)/2$,
\begin{align} \label{complex interpolation hard}
\||V|^{\zeta/2}\chi^2(D)R_0(\lambda\pm\I 0)^{\zeta}V^{\zeta/2}\|_{\mathfrak{S}^{2}}&\leq C(T,\chi,\im\,\zeta)\|V\|_{L^{\frac{2d\re\,\zeta}{d-1+2\re\,\zeta}}}^{\re\,\zeta}.
\end{align}
%where the constant $C$ is independent of $\lambda\in U$. 
Here and in the following, we will use the convention that constants depending on $\im\,\zeta$ grow at most like $\e^{C'|\im\,\zeta|^2}$ for some $C'>0$. In particular, all analytic families we consider here are of admissible growth in the sense of \cite[Theorem 13.1]{MR0246142}. This theorem, together with \eqref{complex interpolation easy}--\eqref{complex interpolation hard} implies that, for all $1\leq a\leq (d+1)/2$,
\begin{align}\label{Schatten bound in proof for lambda pm i epsilon}
\||V|^{1/2}\chi^2(D)R_0(\lambda\pm\I 0)|V|^{1/2}\|_{\mathfrak{S}^{2a}}&\leq C(T,\chi)\|V\|_{L^{\frac{2da}{d-1+2a}}}.
\end{align}
By a change of variables $2a=q(d-1)/(d-q)$, this is the claimed inequality for $2d/(d+1)\leq q\leq (d+1)/2$.
Inequality \eqref{complex interpolation easy} trivially follows from Plancherel's theorem. 
Inequality \eqref{complex interpolation hard} would follow from the pointwise kernel bounds
\begin{align}\label{kernel bound complex powers}
|\chi^2(D)R_0(\lambda\pm\I 0)^{a+\I t}(x-y)|\leq C(t)(1+|x-y|)^{-\frac{d+1}{2}+a}
\end{align}
and the Hardy-Littlewood-Sobolev inequality. It remains to prove \eqref{kernel bound complex powers}. As in the proof of Lemma \ref{lemma resolvent estimate low frequency}, we write $T(\xi)-\lambda$ locally as in \eqref{implicit function theorem}; then
%Without loss of generality we may assume that $e(\xi,\lambda)$ is strictly positive. Then
\begin{align*}
(T(D)-\lambda\pm\I 0)^{-a-\I t}=e(D,\lambda)^{-a-\I t}(D_1-h(D',\lambda)\pm\I 0)^{-a-\I t}.
\end{align*}
Since $(\chi e^{-a-\I t})(D,\lambda)$ is a smoothing operator, \eqref{kernel bound complex powers} would thus follow if we proved
\begin{align}\label{kernel bound complex powers reduced}
|\chi(D)(D_1-h(D',\lambda)\pm\I 0)^{-a-\I t}(x-y)|\leq C(t)(1+|x-y|)^{-\frac{d+1}{2}+a}.
\end{align}
For $a=1$, the estimate \eqref{kernel bound complex powers reduced} follows from \eqref{L1-Linfty}. For $a>1$,
similarly as in the proof of Lemma \ref{lemma resolvent estimate low frequency}, we have
\begin{align*}
\chi(D)(D_1-h(D',\lambda)\pm\I 0)^{-a+\I t}(x)= \int_{\R^{d-1}}\e^{2\pi\I[x'\cdot\xi'+x_1h(\xi',\lambda)]}\psi_t(\xi',x_1)\rd\xi',
\end{align*}
this time with
\begin{align*}
\psi_t(\xi',x_1)=\frac{1}{\Gamma(a-\I t)}\e^{\I(a-\I t)\pi/2}\int_{\R}(x_1-s)_{\pm}^{a-\I t-1}(\mathcal{F}_1\chi(\cdot,\xi'))(-s)\rd s.
\end{align*}
Here, we used an explicit formula for the Fourier transform of $(\xi_1\pm\I 0)^{-a+\I t}$, $a>1$, see e.g.\ \cite{GelfandShilov1977} or \cite[Example 11.3]{Eskin2011}. Note that
$|\Gamma(z)|^{-1}\leq C^{\pi^2|\im z|^2}$.
A Schwartz tails argument then yields
\begin{align*}%\label{Schawartz tails argument implies}
|\partial_{\xi'}^{\alpha}\psi_t(\xi',x_1)|\leq C_{\alpha}\e^{C't^2}(1+|x_1|)^{a-1}\quad\mbox{for all  }\alpha\in\N^d_0.
\end{align*}
Thus, \eqref{kernel bound complex powers reduced} follows from the same stationary phase argument as before. 
% in the proof of Lemma \ref{lemma resolvent estimate low frequency}.

3. We now consider the case $1\leq q< 2d/(d+1)$. We first prove that
\begin{align}\label{Schatten estimate for spectral function}
\||V|^{1/2}\chi^2(D)\rd E(\lambda)|V|^{1/2}\|_{\mathfrak{S}^1}\leq C(T,\chi)\|V\|_{L^1},
\end{align}
where $\rd E(\lambda)$ denotes the spectral measure associated to the operator $T(D)$. To prove this, write $dE(\lambda)=R(\lambda)^*R(\lambda)\rd \lambda$, where $R(\lambda)$ is the Fourier restriction operator to the level set $S_{\lambda}$. As in the proof of \cite[Theorem 2]{FrankSabin14}, we factorize the left hand side as a product of Hilbert-Schmidt operators, $|V|^{1/2}\chi(D)R(\lambda)^*:L^2(S_{\lambda})\to L^2(\R^d)$ and its dual. Since the kernel of this operator is given by $|V(x)|^{1/2}\chi(\xi)\e^{2\pi\I x\cdot\xi}$, one immediately gets
\begin{align*}
\||V|^{1/2}\chi(D)R(\lambda)\|_{\mathfrak{S}^2(L^2(S_{\lambda})\to L^2(\R^d))}^2=\|\chi\|_{L^2(S_{\lambda})}^2\|V\|_{L^1}. 
\end{align*}  
The claim \eqref{Schatten estimate for spectral function} follows from H\"older's inequality in trace ideals. We now write
\begin{align*}
R_0(z)^{b+\I t}=\int_{\R}(\lambda-z)^{-b-\I t}\rd E(\lambda).
\end{align*}
Using the local integrability of $(\cdot-z)^{-b-\I t}$ and \eqref{Schatten estimate for spectral function}, we then obtain, for $0<b<1$,
\begin{align}\label{Schatten bound for re xi less than 1}
\||V|^{1/2}\chi^2(D)R_0(z)^{b+\I t}V|^{1/2}\|_{\mathfrak{S}^1}\leq C(T,\chi,t)(1-b)^{-1}\|V\|_{L^1}.
\end{align}
Next, we observe that, since the kernel of $R_0(\lambda\pm\I 0)^{\frac{d+1}{2}+\I t}$ is bounded, we have 
\begin{align}\label{Schatten bound endpoint a q eq. 1}
\||V|^{1/2}\chi^2(D)R_0(\lambda\pm\I 0)^{\frac{d+1}{2}+\I t}|V|^{1/2}\|_{\mathfrak{S}^2}\leq C(T,\chi,t)\|V\|_{L^1}. 
\end{align}
Consider the analytic family $\zeta\mapsto |V|^{1/2}R_0(z)^{\zeta}|V|^{1/2}$, $b\leq\re\,\zeta\leq a$, with $a=(d+1)/2>1$ and $\mathbf{b}=(1-d\epsilon)/(1-\epsilon)<1$.
Complex interpolation between \eqref{Schatten bound for re xi less than 1} and \eqref{Schatten bound endpoint a q eq. 1} yields 
\begin{align}\label{Schatten bound in S 1+eps}
\||V|^{1/2}\chi^2(D)R_0(\lambda\pm\I 0)|V|^{1/2}\|_{\mathfrak{S}^{1+\epsilon}}\leq C(T,\chi)\epsilon^{-\frac{1-\epsilon}{1+\epsilon}}\|V\|_{L^1} 
\end{align}
for every $\epsilon\in (0,1/2)$. The result for $1\leq q< 2d/(d+1)$ now follows from interpolation between \eqref{Schatten bound in S 1+eps} and the estimate for $2d/(d+1)\leq q\leq (d+1)/2$ already proven (see Lemma \ref{lem: convexity} in the appendix and Figure 1 below).

\begin{figure}[h]\label{Figure}
 \begin{center}
 \psset{unit=.27cm}
\begin{pspicture}(-2,-2)(23,12)
\pspolygon[linestyle=none,fillstyle=solid,fillcolor=gray](10.5,0)(10.5,2.5)(21,10)(21,0)(10.5,0)
\psline(10.5,2.5)(14,5)
\psline[linestyle=dotted](14,5)(21,10)
\psline(10.5,0)(10.5,2.5)
\psline(21,0)(21,10)
\psdot(10.5,2.5)
\psdot(14,5)
\psdot[dotstyle=o](21,10)
\psline[linestyle=solid,arrowsize=5pt,arrows=->](0,0)(22.1,0)\rput[l](22.2,0){$\frac{1}{q}$}
\psline(10.5,-.2)(10.5,.2)\rput[t](10.5,-.5){$\frac{2}{d+1}$}
\psline(14,-.2)(14,.2)\rput[t](14,-.5){$\frac{d+1}{2d}$}
\psline(21,-.2)(21,.2)\rput[t](21,-.5){$1$}
\psline[linestyle=solid,arrowsize=5pt,arrows=->](0,0)(0,11.1)\rput[b](0,11.2){$\frac{1}{\alpha}$}
\psline(-.2,2.5)(.2,2.5)\rput[r](-.5,2.5){$\frac{1}{d+1}$}
\psline(-.2,5)(.2,5)\rput[r](-1,5){$\frac{1}{2}$}
\psline(-.2,10)(.2,10)\rput[t](-1,10){$1$}
\end{pspicture}
\end{center}
\caption{Range of validity of the estimate \eqref{inequality Schatten low frequency}.}
\label{fig:1}
\end{figure}

4. We have proved that \eqref{inequality Schatten low frequency} holds for $z=\lambda\pm\I 0$ for all $\lambda\in\R$. 
Next, we are going to use the Phragm\'en-Lindel\"of maximum principle to prove that the inequality holds for $z\in\C^{\pm}$. We postpone the details to the appendix.
\end{proof}

\begin{remark}\label{remark schatten estimate low frequency}
By inspection of the proofs of Lemma \ref{lemma resolvent estimate low frequency} and Lemma \ref{lemma schatten estimate low frequency}, one checks that the constant $C(T,\chi)$ in~\eqref{inequality Schatten low frequency} depends on $T$ and $\chi$ only through 
\begin{itemize}
\item a lower bound for the Gaussian curvature of $S_{\lambda}$,
\item a lower bound for $|\nabla T|$ on $\supp (\chi)$,
\item The size of $\supp(\chi)$ and finitely many Schwartz semi-norms of $\chi$. 
\end{itemize}
\end{remark}

\begin{proof}[Proof of Theorem \ref{main theorem schatten norms}]
a) Let $K\subset\subset \Omega\subset\subset\Omega'\subset\subset \C\setminus\Lambda_c(H_0)$, and let $\chi$ be a bump function supported on $\Omega'$ and such that $\chi=1$ on $\Omega$. Lemma \ref{lemma schatten estimate low frequency} yields that for all $z\in K\cap\rho(H_0)$, we have 
\begin{align}\label{restriction schatten bound}
\||V|^{1/2}\chi^2(D)R_0(z)|V|^{1/2}\|_{\mathfrak{S}^{\alpha_q}}\leq C\|V\|_{L^q}.
\end{align}
It remains to prove
\begin{align}\label{elliptic Schatten bound}
\||V|^{1/2}(1-\chi^2(D))R_0(z)|V|^{1/2}\|_{\mathfrak{S}^{\alpha_q}}\leq C\|V\|_{L^q}.
\end{align}
For $q>d/s$, the Kato-Seiler-Simon inequality 
yields the better bound
\begin{align}\label{Kato-Seiler-Simon 2}
\||V|^{1/2}(1-\chi^2(D))R_0(z)|V|^{1/2}\|_{\mathfrak{S}^q}\leq C\|V\|_{L^q}.
\end{align}
For $q=d/s$, we use the kernel bound (see \cite[Prop. VI.4.1]{Stein1993})
\begin{align*}
|(1-\chi^2(D))R_0(z)(x-y)|\leq C\left(|x-y|^{s-d}\mathbf{1}\{|x-y|\leq 1\}+\langle x-y\rangle^{-N}\right)
\end{align*}
and the Hardy-Littlewood-Sobolev inequality to show that a better bound than \eqref{elliptic Schatten bound} holds, with the $\mathfrak{S}^2$ norm on the left hand side, provided\footnote{This reflects the conditions in the Hardy-Littlewood-Sobolev inequality.} $d/2<s<d$. For the case $0<s\leq d/2$, we use complex interpolation on the family of operators $|V|^{\zeta/2}(1-\chi^2(D))R_0(z)^{\zeta}|V|^{\zeta/2}$. Similarly as before, we have the kernel bounds 
\begin{equation}\begin{split}\label{kernel bound 1-chi complex powers}
&|(1-\chi^2(D))R_0(z)^{a+\I t}(x-y)|\\
&\leq C(t)\left(|x-y|^{as-d}\mathbf{1}\{|x-y|\leq 1\}+\langle x-y\rangle^{-N}\right),
\end{split}
\end{equation}
provided that $as-d<0$. Note that the constant $C(t)$ grows exponentially since we have
\begin{align*}
\left|\partial_{\xi}^{\alpha}\left[(1-\chi(\xi))(T(\xi)-z)^{-a-\I t}\right]\right|\leq C_{\alpha}\e^{\pi|t|}\langle\xi\rangle^{-as-|\alpha|}.
\end{align*}
Here we choose the branch of the argument function satisfying $-\pi\leq \arg(\cdot)\leq \pi$.
Let $a=d/(2s)+\epsilon$ with $0<\epsilon<d/(4s)$. Then $as-d<0$ and
$0<2(as-d)<d$, so the Hardy-Littlewood-Sobolev inequality together with \eqref{kernel bound 1-chi complex powers} yields 	
\begin{align*}
\||V|^{\zeta/2}(1-\chi^2(D))R_0(z)^{\zeta}|V|^{\zeta/2}\|_{\mathfrak{S}^{2}}
\leq C\|V\|_{L^{d/s}}^{a},\quad\re\,\zeta=a.
\end{align*}
Interpolation with the trivial bound \eqref{complex interpolation easy} yields
\begin{align}\label{elliptic schatten bound very small s}
\||V|^{1/2}(1-\chi^2(D))R_0(z)|V|^{1/2}\|_{\mathfrak{S}^{2a}}
\leq C\|V\|_{L^{d/s}}.
\end{align}
Since $s<d$, we can always choose $\epsilon>0$ such that $2a=d/s+2\epsilon<(d-1)/(s-1)$. Since $q(d-1)/(d-q)=(d-1)/(s-1)$ for $q=d/s$, inequality \eqref{elliptic schatten bound very small s} is therefore better than \eqref{elliptic Schatten bound}.

b) One combines \eqref{restriction schatten bound} with $q=(d+1)/2$ and \eqref{elliptic schatten bound very small s} with  $a=d/(2s)+\epsilon$ as above. The exponents of the Schatten spaces in the resulting estimates are $d+1$ and $2a=d/s+2\epsilon$, respectively. The claim follows from \eqref{embeddings Schatten spaces}. 
\end{proof}

\begin{proof}[Proof of Remark \ref{remark main theorem schatten norms}]
We first consider $H_0=(-\Delta)^{s/2}$, $s\geq 2d/(d+1)$: The claim follows by scaling. More precisely, let $(\sigma_{\delta}f)(x):=\delta^{-d/2}f(x/\delta)$ be the dilation operator and let $A\in\mathfrak{S}^{\alpha}(L^2(X))$, $B:=\sigma_{\delta}A\sigma_{\delta}$. Since $\sigma_{\delta}$ is unitary on $L^2(X)$, it follows that $B\in\mathfrak{S}^{\alpha}(L^2(X))$ and $\|B\|_{\mathfrak{S}^{\alpha}}=\|A\|_{\mathfrak{S}^{\alpha}}$. Taking $\delta=|z|^{1/s}$, $A=|V|^{1/2}R_0(z)|V|^{1/2}$ and $B=|z|^{-1}|V(|z|^{-1/s}\cdot)|^{1/2}R_0(z/|z|)|V(|z|^{-1/s}\cdot)|^{1/2}$ and applying the result of Theorem \ref{main theorem schatten norms} a) with $|z|=1$ yields the claim.

Now consider $H_0=(1-\Delta)^{s/2}-1$, $s\geq 2d/(d+1)$: Write $z=\lambda+\I \epsilon$. In order to find the singularities and decay of $N(z)$, we consider the following cases.
\begin{itemize}
\item[i)] $\delta\leq \lambda\leq \delta^{-1}$ and $0<|\epsilon|\leq 1$,
\item[ii)] $-\delta\leq \lambda\leq 0$ and $0<|\epsilon|\leq 1$,
\item[iii)] $0<\lambda<\delta$ and $\lambda\leq |\epsilon|<1$,
\item[iv)] $0<\lambda<\delta$ and $0<|\epsilon|<\lambda$,
\item[v)] $\lambda<-\delta$ or $|\epsilon|\geq 1$,
\item[vi)] $\delta^{-1}<\lambda$ and $0<|\epsilon|<1$.
\end{itemize} 
Here, $\delta>0$ is some fixed small constant. In case i) Lemma \ref{lemma schatten estimate low frequency} yields a uniform bound $N(z)\leq C(\delta)$.
In cases ii)-iii), pick a bump function $\psi_0$ such that $\psi_0=1$ on $B(0,2\delta)$. Then the estimate for $(1-\psi_0^2(D))R_0(z)$ is uniform in $z$.
For $\xi\in\supp(\psi_0)$ we have $|T(\xi)-z|\geq |z|/2$. Hence, Kato-Seiler-Simon yields 
\begin{align*}
\||V|^{1/2}\psi_0^2(D)R_0(z)|V|^{1/2}\|_{\mathfrak{S}^q}\leq C\|\psi_0^2(T(\cdot)-z)^{-1}\|_{L^q}\|V\|_{L^q},
\end{align*}
In case iv), pick a bump function $\chi$ such that $\supp(\chi)\subset \{1/2\leq|\xi|\leq 3/2\}$ and $\chi=1$ on $\{3/4\leq|\xi|\leq 5/4\}$. For the part $(1-\chi^2(\sqrt{|z|}D))R_0(z)$ the same argument as above applies. For the part $\chi^2(\sqrt{|z|}D)R_0(z)$, we use an approximate scaling argument. Let $\widetilde{T}(\xi)=|z|^{-1}T(|z|^{1/2}\xi)$ and let $\widetilde{R}_0(z)$ be the resolvent of the corresponding multiplier. A straightforward computation shows that
\begin{align}\label{uniform bounds Ttilde}
|\nabla \widetilde{T}(\xi)|\geq \frac{s}{2}-\mathcal{O}(\delta), \quad \Det D^2\widetilde{T}(\xi)\geq s^d-\mathcal{O}(\delta),\quad \xi\in\supp(\chi).
\end{align}
Therefore, for $\delta$ sufficiently small, Lemma \ref{lemma schatten estimate low frequency} and Remark \ref{remark main theorem schatten norms} imply that there is a constant $C_{\epsilon,\delta}$ such that for all $z$ in the region iv), we have
\begin{align*}
\||V|^{1/2}\chi^2(D)\widetilde{R_0}(z/|z|)|V|^{1/2}\|_{\mathfrak{S}^{q(d-1)/(d-q)+\epsilon}}\leq C_{\epsilon,\delta}\|V\|_{L^q},
\end{align*}
where $\epsilon=0$ in case a) or $\epsilon>0$ in case a') of Theorem \ref{main theorem schatten norms}. The same scaling argument as in the homogeneous case then proves the claim.
Case v) is similar to cases ii)-iii), but yields $N(z)=|z|^{\frac{d}{sq}-1}$.
Case vi) is similar to case iv), but here one uses $\widetilde{T}(D)=|z|^{-1}T(|z|^{1/s}D)$ as the rescaled operator and $\chi^2(|z|^{-1/s}D)$ as the localization. Similar lower bounds as in \eqref{uniform bounds Ttilde} are obtained by Taylor expansion of $T(\xi)$ at infinity.

For $H_0\in\{\mathcal{D}_0,\mathcal{D}_1\}$ one uses \eqref{Dirac resolvent}. The modifications are similar as in the proof of Theorem \ref{theorem uniform Sobolev inequality fractional Laplacian}.  
\end{proof}

\section{Proof of Theorem \ref{Main theorem}}

\begin{proof}
Corollary \ref{corollary uniform Sobolev inequality fractional Laplacian} and the Birman-Schwinger principle imply that $H_0+V$ has no eigenvalues $z\in K\cap\rho(H_0)$ if $\|V\|$ is sufficiently small\footnote{$\|V\|$ was defined in \eqref{def. of ||V||}}. 

To prove \eqref{sum of eigenvalues converges} we define the holomorphic function\footnote{For a proof that $h$ is holomorphic see e.g.\ \cite[Lemma 12]{FrankSabin14}.}
\begin{align}\label{definition of holomorphic h}
\C\setminus[0,\infty)\ni z\mapsto h(z)=\Det_{\lceil \alpha\rceil}(I+V^{1/2}R_0(z)|V|^{1/2})\in\C,
\end{align}
where $\lceil \alpha\rceil$ is the smallest integer which is $\geq \alpha$ and $\alpha$ is the exponent in the Schatten space  $\mathfrak{S}^{\alpha}$ appearing in the bounds of Theorem \ref{main theorem schatten norms}. We remark that in the case of the Dirac operator, $V^{1/2}=|V|^{1/2}U$, where $U$ is the partial isometry in the polar decomposition of $V=|V|U$. By unitary invariance of the Schatten ideals $\mathfrak{S}^{\alpha}$, we may neglect $U$ in the following estimates. 
%The eigenvalues $z\in\rho(H_0)$ of $H_0+V$ are the zeros of $h$.
By \cite[Lemma XI.9.22]{DS2} and Theorem \ref{main theorem schatten norms} we have the bound 
\begin{align}\label{log bound in terms of Schatten norm}
\log|h(z)|\leq \Gamma_{\alpha}\||V|^{1/2}R_0(z)|V|^{1/2}\|_{\mathfrak{S}_{\alpha}}^{\alpha}\leq \Gamma_{\alpha}C^{\alpha}  N(z)^{\alpha}\|V\|^{\alpha}.
\end{align}
We need to map $\rho(H_0)$ conformally onto the unit disk $\mathbb{D}\subset\C$. Due to the special form of $\rho(H_0)$ 
we can write down such a map explicitly. In the case $H_0=\mathcal{D}_0$, we have $\rho(H_0)=\C^+\cup\C^-$, so we need two conformal maps $\varphi^{\pm}$ in this case. We can choose e.g.
\begin{align}\label{phitilde}
\varphi^+:\C^+\to\mathbb{D},\quad \varphi^+(z)=\frac{z-\I}{z+\I},
\end{align}
and an obvious modification for $\varphi^-$. If $H_0=(-\Delta)^{s/2}$ or $H_0=(1-\Delta)^{s/2}-1$, we can take $\varphi^+(\sqrt{\cdot})$, where $\sqrt{\cdot}$ is the principal branch of the square root on $\C\setminus[0,\infty)$. If $H_0=\mathcal{D}_1$, we compose the last map with $z\mapsto (z-1)/(z+1):\rho(H_0)\to\C\setminus[0,\infty)$. In all cases, the resulting map is denoted by $\psi_0:\rho(H_0)\to\mathbb{D}$. We will need a normalized version of this map. Fix $z_0\in\rho(H_0+V)$, $|z_0|>1$, and set 
 $\widetilde{z_0}=\psi_0(z_0)\in\mathbb{D}$. With the normalization
\begin{align}\label{normalization map}
\nu:\mathbb{D}\to\mathbb{D},\quad\nu(w)=\frac{w+\widetilde{z_0}}{1+\widetilde{z_0}w},
\end{align} 
we then define the map
\begin{align}\label{psi}
\psi=\nu^{-1}\circ \psi_0:\rho(H_0)\to \mathbb{D}
\end{align}
which has the property $\psi(z_0)=0$. In all cases, $\psi$ is a $C^{\infty}(K)$-diffeomorphism, i.e.\ $\psi$ and $\psi^{-1}$ are continuous on $K$, together with all their derivatives. In particular, since $K$ is compact, all these functions are bounded on $K$.

We note that a point $z_0$ as above always exists, but it may depend on $V$ instead of on $\|V\|$ only. This is due to the fact that $N(z)$ in Theorem \ref{main theorem schatten norms} may not tend to zero at infinity (and indeed may diverge).
This is only the case if we assume Assumption~\ref{potential} a) with $q=d/s$ or Assumption~\ref{potential} b). To find such a $z_0$, we decompose $V=V_1+V_2$, where $V_1=V\chi\{x:|V(x)|\geq \rho\}$ and $\rho$ to be chosen sufficiently large. Then $V_2\in L^{\infty}$, and thus $\sigma(H_0+V)$ lies in a $\rho$-neighborhood of $\sigma(H_0+V_1)$. By Chebyshev's inequality,
\begin{align*}
|\{x:|V(x)|\geq \rho\}|\leq \frac{\|V\|_{L^{d/s}}^{d/s}}{\rho^{d/s}}.
\end{align*}
Dominated convergence shows that $\|V_1\|_{L^{d/s}}$ tends to zero as $\rho\to\infty$. Choosing $\rho$ sufficiently large, we infer as in the beginning of the proof that $\sigma(H_0+V_1)\subset\sigma(H_0)$.\footnote{We only proved the absence of eigenvalues. However, by standard arguments $H_0+V$ can only have discrete spectrum outside $\sigma(H_0)$.}
We can thus pick any $z_0$ outside a $\rho$-neighborhood of $\sigma(H_0)$. If we assume Assumption \ref{potential} a) with $d/s<q\leq (d+1)/2$, then Theorem \ref{thm bounds on individual eigenvalues} shows that choosing a $z_0\in\rho(H_0)$ such that
\begin{align}\label{choice z0}
|z_0|\geq C \|V\|_{L^q}^{sq/(sq-d)},
\end{align}
with $C$ sufficiently large, we will have $z_0\in \rho(H_0+V)$ and thus $h(z_0)\neq 0$. The lower bound in \eqref{choice z0} only depends on $d,q,s$ via the constant in Theorem \ref{eq eigenvalue bound fractional Laplcian a) homog}. With this choice of $z_0$, let us define a holomorphic function
on the unit disk,
\begin{align}\label{holomorphic g}
g:\mathbb{D}\to \C,\quad g(w)=\frac{h(\psi^{-1}(w))}{h(z_0)}.
\end{align} 
Note the normalization $g(0)=1$. By \eqref{log bound in terms of Schatten norm}, we have
\begin{equation}
\begin{split}
\log|g(w)|&\leq C N(\psi^{-1}(w))^{\alpha}\|V\|^{\alpha}-\log|h(z_0)|\label{log bound g}\\
&\leq C(V)\prod_{w_c\in \psi(\Lambda_c(H_0)\cup\{\infty\})}|w-w_c|^{-\mu(w_c)}
\end{split}
\end{equation}
where $\mu(w_c)\geq 0$ depend on $d,s,q$ and the conformal map $\psi$ and could be computed from the estimates in Remark~\ref{remark main theorem schatten norms}). We will not need the precise values for the time being.\footnote{Some concrete examples are given in Section \ref{section further results}}
Note that the constant $C(V)$ depends on $V$ via $\|V\|$ and $z_0$.

We now invoke a theorem of Borichev, Golinskii and Kupin \cite[Thm. 03]{BorichevEtAl2009}: Assume that $g:\mathbb{D}\to\C$ is analytic, $|g(0)|=1$ and that there exist $M>0$, $\mathbf{b}=\{b_i\}_{i=1}^N\subset\partial\mathbb{D}$ and $\mathbf{k}=\{k_i\}_{i=1}^N\subset [0,\infty)$ such that $\log|g(w)|\leq M\prod_{i=1}^N|w-b_i|^{-k_i}$. Then, for any $\tau>0$, there exists $A_{\kappa,\tau}>0$ such that
\begin{align}\label{Borichev Golinski Kupin}
\sum_{g(w)=0}(1-|w|)\prod_{i=1}^N|w-b_i|^{(k_i-1+\tau)_+}\leq A_{\kappa,\tau}M.
\end{align}
Here and in the following, all zeros are counted with multiplicity.
Using \eqref{log bound g}, we apply this theorem with $\mathbf{b}=\psi(\Lambda_c(H_0)\cup\{\infty\})$, $\vec{\kappa}=\{\mu(w_c):w_c\in b\}$ and with $M=C(V)$. The result is that
\begin{align}\label{second last step in main thm}
\sum_{g(w)=0}(1-|w|)\prod_{w_c\in \psi(\Lambda_c(H_0)\cup\{\infty\})}|w-w_c|^{(\mu(w_c)-1+\tau)_+}\leq  A_{\kappa,\tau} C(V).
\end{align}
Koebe's distortion theorem \cite[Corollary 1.4]{MR1217706} implies that
\begin{align}\label{Koebe distortion}
\frac{1}{4}(1-|w|)|\psi'(w)|\leq \dist(\psi(w),\sigma(H_0))\leq 2(1-|w|)|\psi'(w)|
\end{align}
for all $w\in\mathbb{D}$. In particular, we have that $(1-|w|)\approx \dist(\psi(w),\sigma(H_0))$ for $z\in K$ since 
\begin{align*}
\frac{1}{C(K)}\leq |\psi'(w)|\leq C(K),\quad w\in \psi^{-1}(K).
\end{align*}
Writing \eqref{second last step in main thm} in terms of $z$ instead of $w$ and restricting summation to $z\in K$, we thus obtain
\begin{align}\label{last step in main thm}
\sum_{z\in \sigma_{\rm d}(H_0+V)\cap K}\dist(z,\sigma(H_0))\leq C(K,V).
\end{align}
\end{proof}

\begin{remark}
As already pointed out, we only need very rough bounds for $N(z)$ in the above proof. It might be asked why we need bounds $N(z)$ for $z\notin K$; after all, the summation is only over $z\in K$. Suppose $K'\subset\subset\Omega\subset\subset K$, where $\Omega\subset \rho(H_0)$ is a domain. We could use the Riemann mapping theorem to map $\psi:\Omega\to\mathbb{D}$ conformally. The bound for $\log|g(w)|$ in \eqref{log bound g} would then be uniform and \eqref{last step in main thm} would hold with $K$ replaced by $K'$, with a bound depending on $\sup_{z\in K'}|\psi'(z)|$. But since $\dist(K',\sigma(H_0))>0$, this is obvious as the sum is finite. To get a non-trivial result, we need $\dist(\Omega,\sigma(H_0))=0$, and hence $\partial\Omega$ could be non-smooth. In this case, it is generally not true that $\psi$ extends smoothly up to the boundary. 
\end{remark}

\section{Further results}\label{section further results}

In the following, we consider some refinements of Theorem \ref{Main theorem} in certain special cases. The character of the results presented here is exemplary rather than exhaustive. In particular, we extend results of \cite{Frank11}, \cite{FrankSabin14} and \cite{LaSa09} to fractional Schr\"odinger and Dirac operators.

\subsection{Bounds on individual eigenvalues}

\begin{theorem}\label{thm bounds on individual eigenvalues}
a) Let $H_0=(-\Delta)^{s/2}$ and $2d/(d+1)\leq s<d$, and assume $V\in L^q(X)$ with $d/s\leq q\leq (d+1)/2$. Then any eigenvalue $z\in\C\setminus[0,\infty)$ of $H_0+V$ satisfies
\begin{align}\label{eq eigenvalue bound fractional Laplcian a) homog}
|z|^{q-\frac{d}{s}}\leq C\|V\|_{L^q}^q. 
%\quad\mbox{if  } H_0=I_s,\label{eq eigenvalue bound fractional Laplcian a) homog}
%|z|^{q-\frac{d}{s}}\leq C(1+\|V\|_{L^q}^q) \quad\mbox{if  } H_0=\mathcal{J}_s\label{eq eigenvalue bound fractional Laplcian a) inhomog}.
\end{align}
b) Let $H_0=T(D)$ with $T(D)$ as in Assumption \ref{unperturbed} and with $0<s<d$, and assume $V\in L^q(X)$ with $d/s\leq q$. Then any eigenvalue $z\in\C\setminus\sigma(H_0)$ of $H_0+V$ satisfies
\begin{align}\label{eq eigenvalue bound fractional Laplcian b) homog}
(|\im z|/|\re z|)^{d/s-1}|\im z|^{q-d/s}\leq C\|V\|_{L^q}^q
%\quad\mbox{if  } H_0\in\{I_s,\mathcal{D}_0\},\label{eq eigenvalue bound fractional Laplcian b) homog}
%\\
%(|\im z|/|\re z|)^{d/s-1}|\im z|^{q-d/s}\leq C(1+\|V\|_{L^q}^q) 
%\quad\mbox{if  } H_0\in\{\mathcal{J}_s,\mathcal{D}_1\}%\label{eq eigenvalue bound fractional Laplcian b) inhomog}.
\end{align}
\end{theorem}

\begin{remark}
Inequality \eqref{eq eigenvalue bound fractional Laplcian a) homog} for $s=2$ was proved in \cite{Frank11}. Inequality \eqref{eq eigenvalue bound fractional Laplcian b) homog} is much less precise, but it still gives better bounds for $z$ close to $\sigma(H_0)$ compared to the often used estimate
\begin{align*}
1\leq \||V|^{1/2}R_0(z)|V|^{1/2}\|\leq C\|(T(\cdot)-z)^{-1}\|_{L^q}\|V\|_{L^q},
\end{align*}
which is an easy consequence of the Kato-Seiler-Simon inequality if $q>d/s$.
\end{remark}

\begin{proof}
The theorem is proved using the Birman-Schwinger principle, as before. 
Inequality \eqref{eq eigenvalue bound fractional Laplcian a) homog} follows from Corollary \ref{corollary uniform Sobolev inequality fractional Laplacian}, scaling and H\"older's inequality. Inequality~\eqref{eq eigenvalue bound fractional Laplcian b) homog} would be a consequence of the resolvent bound
\begin{align}\label{eq resolvent bound fractional Laplcian b)}
\left\|\left(R_0(\lambda+\I\mu\right)f\right\|_{L^{p'}}\leq C |\lambda|^{\frac{d-s}{s}\left(\frac{1}{p}-\frac{1}{p'}\right)}|\mu|^{\left(\frac{1}{p}-\frac{1}{p'}\right)-1}\|f\|_{L^{p}},
\end{align}
where $1/p-1/p'=1/q$.
By scaling, we may assume $|\lambda|=1$. We may also assume that $|\mu|\leq 1$, otherwise the bounds are uniform. On the one hand, we have the Sobolev inequality~\eqref{high frequency bound case b in the proof of LpLp'} for $(1-\chi^2(D))R_0(z)$, which is better than \eqref{eq resolvent bound fractional Laplcian b)}\footnote{Inequality \eqref{high frequency bound case b in the proof of LpLp'} is only stated for $1/p-1/p'=s/d$. The general case is obtained by interpolation with the $L^2\to L^2$ bound.}. On the other hand, for $\chi^2(D)R_0(z)$ one interpolates between the bounds
\begin{align*}
\|\chi^2(D)R_0(z)\|_{L^1\to L^{\infty}}&\leq C,\\
\|\chi^2(D)R_0(z)\|_{L^{2}\to L^2}&\leq |\mu|^{-1}.
\end{align*}
The first follows from comparison with the Hilbert transform \eqref{Fourier transform of Hilbert transform}, the second is obvious.
\end{proof}

\begin{theorem}\label{theorem purely imaginary potentials}
Let $H_0=(-\Delta)^{s/2}$ or $H_0=\mathcal{D}_0$, $d/(d+1)\leq s<d$, and let $V\in L^q(X)$ with
\begin{align*}
\begin{cases}
\frac{d}{2s}\leq q\leq \frac{d+1}{2}\quad&\mbox{if   }2s<d,\\
1<q\leq \frac{d+1}{2}\quad&\mbox{if   }2s=d,\\
1\leq q\leq \frac{d+1}{2}\quad&\mbox{if   }2s>d.
\end{cases}
\end{align*}
Assume that $V=\I W$ with $W\geq 0$. Then every eigenvalue $z\in \C\setminus\sigma(H_0)$ of $H_0+ V$ satisfies
\begin{align}\label{bound purely imaginary}
|z|^{2q-d/s}|\im z|^{-q}\leq C\|V\|_q^q.
\end{align}
\end{theorem}

\begin{remark}
Theorem \ref{theorem purely imaginary potentials} generalizes \cite[Thm. 9]{LaSa09} for Schr\"odinger operators with purely imaginary potentials in three dimensions. Note in particular that the case $s=1$ is included, in contrast to the general situation in Theorem \ref{thm bounds on individual eigenvalues} a).
\end{remark}

\begin{proof}
Suppose $z$ is an eigenvalue of $H_0+\I W$. By the Birman-Schwinger principle, this is equivalent to $1$ being an eigenvalue of the operator
\begin{align}\label{purely imaginary Birman Schwinger}
Q(z):=\I\sqrt{W}R_0(z)\sqrt{W},
\end{align}
i.e.\ there exists $f\in L^2(\R^d)$, $\|f\|=1$, such that $Q(z)f=f$. It follows that
\begin{align}\label{reQff}
(\re\, Q(z))f,f)=1.
\end{align}
On the other hand, since 
\begin{align*}
\re\, Q(z)=\frac{1}{2}\sqrt{W}\,\im(R_0(z))\sqrt{W},
\end{align*}
we have
\begin{align}\label{norm reQ}
\|\re\, Q(z)\|\leq  \frac{1}{2}\|V\|_{L^{\frac{p}{2-p}}}\|\im(R_0(z))\|_{L^p(\R^d)\to L^{p'}}.
\end{align}
In view of the equality $\im(R_0(z))=(\im z) R_0(z)R_0(\overline{z})$, it remains to prove the following resolvent estimate,
\begin{align}\label{resolvent estimate imaginary part}
\|R_0(z)R_0(\overline{z})\|_{L^p\to L^{p'}}\leq C|z|^{\frac{d}{s}\left(\frac{1}{p}-\frac{1}{p'}\right)-2}
\end{align}
where $\frac{1}{p}-\frac{1}{p'}=\frac{1}{q}$ and $q$ as in the assumptions of the theorem. By scaling, the proof of \eqref{resolvent estimate imaginary part} reduces to the case $|z|=1$. Let $\chi$ be the same bump function as before.
The estimate of $(1-\chi^2(D))R_0(z)$ is as in~\eqref{high frequency bound case b in the proof of LpLp'}, but with $2s$ instead of $s$. The estimate for $\chi^2(D)R_0(z)$ is further reduced to the cases $z=1\pm\I 0$ and $z=-1\pm\I 0$ by the Phragm\'en-Lindel\"of maximum principle (the proof is similar and somewhat easier than the one given in the appendix). The case $z=-1\pm\I 0$ is again handled by the Sobolev inequality \eqref{high frequency bound case b in the proof of LpLp'}. For $z=1\pm\I 0$, one notes that $\im R_0(1\pm\I 0)f=c\widehat{\rd\sigma}_{S^{d-1}}*f$, where $\rd\sigma$ is the usual surface measure on the unit sphere $S^{d-1}$.
The result thus follows from the $TT^*$ version of the Stein Tomas theorem \cite{Tomas1975}.
\end{proof}

\begin{remark}
The bounds \eqref{eq eigenvalue bound fractional Laplcian a) homog} and \eqref{bound purely imaginary} can be extended to all eigenvalues of $H_0+V$, see \cite[Prop. 3.1]{FrankSimon2015}.
\end{remark}

\subsection{Weighted eigenvalue sums}

\begin{theorem}\label{thm eigenvalue sums big s}
Let $d\geq 2$, $H_0=(-\Delta)^{s/2}$, $s\geq 2d/(d+1)$, and let $V\in L^{q}(X)$, with $d/s<q\leq (d+1)/2$. Then
\begin{align}\label{eigenvalue sums big s}
\sum_{z\in\sigma_{\rm d}(H_0+V)}|z|^{-(1-\delta)/2}\dist(z,\sigma(H_0))\leq C_{q,d,s,\delta}\|V\|^{\frac{(1+\delta)sq}{2(sq-d)}}_{L^q},
\end{align}
where
\begin{align*}
\begin{cases}
\delta\geq 0\quad&\mbox{if    } (d,s,q)\in A_1\cup A_2\cup A_3\cup A_4\cup A_5,\\
\delta>\frac{2s-2d^2-ds-qs+2qs}{(d-q)s} \quad&\mbox{if    }(d,s,q)\notin A_1\cup A_2\cup A_3\cup A_4\cup A_5.
\end{cases}
\quad 
\end{align*}
Here and in the following, each eigenvalue $z\in\sigma_{\rm d}(H_0+V)$ is counted according to its algebraic multiplicity, and
$A_j\subset \N\times\R_+\times\R_+$, $j=1,\ldots,5$, are the sets
\begin{align*}
A_1&:=\{2\}\times\left(\frac{4}{3},\frac{8}{5}\right)\times\left(\frac{2}{s},\frac{3}{2}\right],\\
A_2&:=\{2\}\times\left\{\frac{8}{5}\right\}\times\left(\frac{5}{4},\frac{3}{2}\right),\\
A_3&:=\{2\}\times\left(\frac{8}{5},2\right)\times\left(\frac{2}{s},\frac{4+2s}{3s}\right),\\
A_4&:=\{d\in\N:d\geq 3\}\times\left(\frac{2d}{1+d},\frac{4d}{1+2d}\right)\times\left(\frac{d}{s},\frac{1+d}{2}\right],\\
A_5&:=\{d\in\N:d\geq 3\}\times\left[\frac{4d}{1+2d},d\right)\times\left(\frac{d}{s},\frac{2d^2-2d+ds}{2ds-s}\right).
\end{align*}
\end{theorem}

\begin{remark}
For $H_0=-\Delta$ this result was proved in \cite[Thm. 16]{FrankSabin14}. Note that the relevant set for this case in $d\geq 3$ dimensions is $A_5$. Our result does not cover the case $d=2$ in \cite[Thm. 16]{FrankSabin14} since we restricted our attention to $s<d$. 
%For $H_0=(1-\Delta)^{s/2}-1$ the estimate \eqref{eigenvalue sums big s} holds with $s$ replaced by $2$. The reason is the singularity of $N(z)$ in Remark \ref{remark main theorem schatten norms} at $z=0$. 
\end{remark}

\begin{proof}
The strategy of the proof is similar as that of Theorem \ref{Main theorem}, see also \cite[Thm. 16]{FrankSabin14}. From Theorem \ref{main theorem schatten norms}, Remark \ref{remark main theorem schatten norms} and \eqref{log bound in terms of Schatten norm} we infer that
\begin{align}\label{log h(z) for fractional big s}
\log|h(z)|\leq \Gamma_{\alpha_q}C^{\alpha_q}|z|^{\frac{\alpha_q d}{sq}-\alpha_q}\|V\|_{L^q}^{\alpha_q},
\end{align}
where we recall that
\begin{align*}
\begin{cases}&\alpha_q=\frac{(d-1)q}{d-q}\quad\mbox{if   } \frac{2d}{d+1}\leq s\leq \frac{d+1}{2} \mbox{  and } \frac{d}{s}\leq q\leq \frac{d+1}{2},\\
&\alpha_q>\frac{(d-1)q}{d-q}\quad\mbox{if   } \frac{d+1}{2}<s<d \mbox{  and } \frac{d}{s}\leq q\leq \frac{2d}{d+1}.\end{cases}
\end{align*}
Since $\sigma(H_0)=[0,\infty)$, the (normalized) conformal map $\psi:\rho(H_0)\to\mathbb{D}$ is given by 
\begin{align}\label{def. inverse conformal map fractional}
\psi(z)=\frac{\sqrt{z}-\sqrt{z_0}}{\sqrt{z}+\sqrt{z_0}},\quad z\in\rho(H_0)=\C\setminus[0,\infty),
\end{align}
and its inverse is
\begin{align}\label{def. conformal map fractional}
\psi^{-1}(w)=z_0\left(\frac{1+w}{1-w}\right)^2,\quad w\in\mathbb{D}.
\end{align}
By \cite[Thm. 9.2 (c)]{Simon2005}\footnote{with the modification for non-integer $\alpha_q$ mentioned in \cite{FrankSabin14}} we have
\begin{align}\label{Simon trace ideals Thm 9.2 (c)}
|h(z)-1|\leq M\exp(\Gamma_q(M+1)^{\alpha})
\end{align} 
where $\Gamma_q>0$ and $M:=\||V|^{1/2}R_0(z)|V|^{1/2}\|_{\mathfrak{S}^{\alpha_q}}^{\alpha_q}$. We choose $\epsilon_q>0$ such that $x\exp(\Gamma_q(x+1)^{\alpha_q})\leq 1/2$ for all $x\in [0,\epsilon_q]$ and then choose $z_0\in \C\setminus[0,\infty)$ such that
\begin{align}\label{def. of z0}
C^{\alpha_q}|z_0|^{\frac{\alpha_q d}{sq}-\alpha_q}\|V\|_{L^q}^{\alpha_q}= \epsilon_q.
\end{align}
In view of Theorem \ref{thm bounds on individual eigenvalues}, by taking $\epsilon_q$ small enough, we can also arrange that 
\begin{align}\label{zo bigger than z}
|z_0|>\sup\set{|z|}{z\in\sigma_{\rm d}(H_0+V)}.
\end{align}
It then follows, again from Theorem \ref{main theorem schatten norms} and Remark \ref{remark main theorem schatten norms}, that $|h(z_0)-1|\leq 1/2$. In particular, we have that $\log|h(z_0)|\geq -\log 2$. Combining \eqref{log h(z) for fractional big s} and \eqref{def. of z0} we thus obtain the estimate
\begin{align}\label{log g for big s}
\log|g(w)|\leq \Gamma_{\alpha_q}C^{\alpha_q}\frac{|z_0|^{\frac{\alpha_q d}{sq}-\alpha_q}\|V\|_{L^q}^{\alpha_q}}{|1+w|^{2\alpha_q-\frac{2\alpha_q d}{sq}}}-\log|h(z_0)|\leq \frac{C_q'}{|1+w|^{2\alpha_q-\frac{2\alpha_q d}{sq}}},
\end{align}
with a constant $C_q'$ independent of $V$ and $z_0$. Here, $g:\mathbb{D}\to \C$ is the analytic function defined in \eqref{holomorphic g} with $\psi^{-1}$ given explicitly by \ref{def. conformal map fractional}. We apply \eqref{Borichev Golinski Kupin} with $\mathbf{b}=\{-1\}$, $\vec{\kappa}=\{\kappa_{q,d,s}:=2\alpha_q(sq-d)/(sq)\}$, $M=C_q'$ and $\tau=1-\kappa_{q,d,s}+\delta$, where $\delta$ was defined in the statement of the theorem. The choice of $\delta$ was made in such a way that $\tau>0$, which is necessary for the result of \cite{BorichevEtAl2009} to be applicable. This may be checked by a direct computation in the case that $\alpha_q=(d-1)q/(d-q)$. In order to have $\delta\geq 0$, we need $2\alpha_q(sq-d)/(sq)<1$, and this inequality holds in the union of the sets $A_j$, $j=1,\ldots,5$, under the given restriction on $s$ and $q$. A continuity argument then implies that the same is true for $\alpha_q=(d-1)q/(d-q)+\epsilon$ if $\epsilon>0$ is chosen small enough.
The result is that
\begin{align*}
\sum_{g(w)=0}(1-|w|)|1+w|^{\delta}\leq A_{\kappa,\delta}C_q'.
\end{align*}
After a straightforward application of Koebe's distortion theorem, one obtains 
\begin{align*}
\sum_{z\in\sigma_{\rm d}(H_0+V)}\frac{\dist(z,\sigma(H_0))\sqrt{|z_0|}}{2^{1-\delta}|z|^{(1-\delta)/2}(\sqrt{|z|}+\sqrt{|z_0|})^{2+\delta}}\leq A_{\kappa,\delta}C_q'.
\end{align*}
We refer to \cite{FrankSabin14} and \cite{DemuthEtAl2009} for the details. Using \eqref{zo bigger than z}, it follows that
\begin{align*}
\sum_{z\in\sigma_{\rm d}(H_0+V)}|z|^{-(1-\delta)/2}\dist(z,\sigma(H_0))\leq  8 A_{\kappa,\delta}C_q'|z_0|^{(1+\delta)/2}.
\end{align*}
The claim follows from the definition \eqref{def. of z0} of $z_0$.
\end{proof}

\begin{theorem}\label{Theorem Dirac massless sum}
Let $d\geq 2$, $H_0=\mathcal{D}_0$, and let $V\in L^d(X)\cap L^{(d+1)/2}(X)$. Then, for any $\tau>0$, we have the estimate
\begin{align}\label{eigenvalue sums massless Dirac}
\sum_{z\in\sigma_{\rm d}(H_0+V)}\dist(z,\sigma(H_0))(1+|z|)^{-d-\tau}\leq C(V,\tau).
\end{align}
\end{theorem}

\begin{proof}
We prove the result for the part of the sum with $z\in\C^+$; the part with $z\in\C^-$ is treated similarly. As in the proof of Theorem \ref{Main theorem}, we have the estimate~\eqref{log bound g}, where $\psi=\nu^{-1}\circ\varphi^+$ with $\varphi^+$ and $\nu$ as in \eqref{phitilde} and \eqref{normalization map}.
%Abusing notation, we denote functions and variables by the same symbol, e.g.\ $\nu=\nu(w)$. Then
%\begin{align*}
%\nu=\frac{z-\I}{z+\I},\quad z=-\I\frac{\nu+1}{\nu-1}.
%\end{align*}
%Once $z_0$ has been fixed, the normalization map \eqref{normalization map} extends to a $C^{\infty}(\overline{\mathbb{D}})$- diffeomorphism with
Using the bounds for $N(z)$ in Remark \ref{remark main theorem schatten norms}, estimate \eqref{log bound g} can be refined in the following way,
\begin{align*}
\log|g(w)|&\leq C N((\varphi^+)^{-1}\circ \nu(w))^{\alpha}\|V\|^{\alpha}-\log|h(z_0)|\\
&=C\left|\frac{\nu(w)+1}{\nu(w)-1}\right|^{\alpha\frac{d-1}{d+1}}\|V\|^{\alpha}-\log|h(z_0)|\\
&\leq 2^{\alpha\frac{d-1}{d+1}}(C\|V\|^{\alpha}-\log_-|h(z_0)|)|\nu(w)-1|^{-\alpha\frac{d-1}{d+1}},
\end{align*}
where $\alpha=\max\{d+1,d+\epsilon\}=d+1$. From the definition of the map $\nu$ we see that
\begin{align*}
\frac{1-|\widetilde{z}_0|}{1+|\widetilde{z}_0|}|w-1|\leq |\nu(w)-1|\leq \frac{1+|\widetilde{z}_0|}{1-|\widetilde{z}_0|}|w-1|.
\end{align*}
Recall that $\widetilde{z}_0=\psi_0(z_0)$ and $z_0$ depend on $V$. Combining the previous estimates, we thus get that
\begin{align*}
\log|g(w)|\leq C(V)|w-1|^{-(d-1)},
\end{align*}
for some $C(V)$ depending on $V$ via $\|V\|^{d+1}$ and $z_0$. Applying \eqref{Borichev Golinski Kupin} with $\mathbf{b}=\{1\}$ and $\mathbf{k}=\{d-1\}$, we obtain that
\begin{align}\label{sum gw Dirac massless}
\sum_{g(w)=0}(1-|w|)|w-1|^{d-2+\tau}\leq C(V,\tau)
\end{align}
for all $\tau>0$.
By the Koebe distortion theorem, we have (by slight abuse of notation)
\begin{align*}
(1-|w|)\approx\left|\frac{\rd w}{\rd z}\right|\dist(z,\sigma(H_0))
=\left|\frac{\rd w}{\rd \nu}\right|\left|\frac{\rd \nu}{\rd z}\right|\dist(z,\sigma(H_0))
\approx \frac{\dist(z,\sigma(H_0))}{|z+\I|^2},
\end{align*}
where we used that
\begin{align*}
\frac{\rd \nu}{\rd z}=\frac{\rd \varphi^+(z)}{\rd z}=\frac{2\I}{(z+\I)^2},\quad z\in\overline{\C^+}
\end{align*}
and
\begin{align}\label{normalization is a diffeo}
\frac{1-|\widetilde{z_0}|^2}{4}\leq |\nu'(w)|\leq \frac{2}{(1-|\widetilde{z_0}|)^2},\quad w\in \overline{\mathbb{D}}.
\end{align}
In addition, since
\begin{align*}
|w-1|\approx |\nu-1|\approx\frac{1}{|z+\I|}\geq \frac{1}{|z|+1},
\end{align*}
the result follows from \eqref{sum gw Dirac massless}.
%\begin{align*}
%\sum_{h(z)=0,\, z\in\C^+}\dist(z,\sigma(H_0))(1+|z|)^{-d-\tau}\leq C(V,\tau).
%\end{align*}
\end{proof}

\begin{theorem}\label{thm eigenvalue sums massive Dirac}
Let $d\geq 2$, $H_0=\mathcal{D}_1$, and let $V\in L^d(X)\cap L^{(d+1)/2}(X)$. Then, for any $\tau>0$, we have the estimate
\begin{align}\label{eigenvalue sums massive Dirac}
\sum_{z\in\sigma_{\rm d}(H_0+V)}\dist(z,\sigma(H_0))|z^2-1|^{\frac{d-1}{2}+\tau}(1+|z|)^{-2d+1-\tau}\leq C(V,\tau).
\end{align}
\end{theorem}

\begin{proof}
Abusing notation slightly, we write $z=z(\nu)$ and $\nu=\nu(z)$. Here,
\begin{align}\label{z=znu}
z=-\frac{2\nu}{1+\nu^2},\quad z\in \rho(H_0),\quad\nu\in\mathbb{D}.
\end{align}
The map $z\mapsto \nu(z)$ is constructed as outlined in the proof of Theorem \ref{Main theorem}, i.e.\ by composition of $\zeta\mapsto \nu=(\sqrt{\zeta}-\I)/(\sqrt{\zeta}+\I)$ with $z\mapsto\zeta=(z-1)/(z+1)$. Using Remark \eqref{remark main theorem schatten norms} to estimate $N(z)$ and observing that
\begin{align*}
|1-z^2|=\frac{|1+\nu|^2|1-\nu|^2}{|1+\I\nu|^2|1-\I\nu|^2},\quad |z|\leq \frac{2}{|1+\I\nu||1-\I\nu|},
\end{align*}
the estimate \eqref{log bound g}
% with $w_c^1=1$, $w_c^2=-1$, and with $\mu_c^i=1/2$, $\mu_{\infty}=(d-1)/(d+1)$ (see Remark \ref{remark main theorem schatten norms}) 
can be refined as follows,
\begin{align*}
\log|g(w)|&\leq C(V)|1-z^2|^{-\frac{\alpha}{2}}+(1+|z|)^{\alpha\frac{d-1}{d+1}}\\
&\leq C(V)|1-\nu|^{-\alpha}|1+\nu|^{-\alpha}|1-\I\nu|^{-\alpha\frac{d-1}{d+1}}|1+\I\nu|^{-\alpha\frac{d-1}{d+1}}\\
&\leq C(V)|w-w_1|^{-\alpha}|w-w_2|^{-\alpha}
|w-w_3|^{-\alpha\frac{d-1}{d+1}}|w-w_4|^{-\alpha\frac{d-1}{d+1}} 
\end{align*}
with $\alpha=d+1$, and where we have set
\begin{align*}
w_1=\nu^{-1}(1),\quad w_2=\nu^{-1}(-1),\quad w_3=\nu^{-1}(\I),\quad w_4=\nu^{-1}(-\I).
\end{align*} 
We apply \eqref{Borichev Golinski Kupin} with $\mathbf{b}=\{1,-1,\I,-\I\}$ and $\mathbf{k}=\{d+1,d+1,d-1,d-1\}$. Due to \eqref{normalization is a diffeo}, this yields that
\begin{equation}
\begin{split}
&\sum_{g(w)=0}(1-|w|)|w-w_1|^{\alpha-1+\tau}|w-w_2|^{\alpha-1+\tau}\label{sum gw Dirac massive}\\
&\quad\times|w-w_3|^{\alpha\frac{d-1}{d+1}-1+\tau}|w-w_4|^{\alpha\frac{d-1}{d+1}-1+\tau}\leq C(V,\tau),
\end{split}
\end{equation}
for all $\tau>0$.
The proof is completed by using the following estimates in \eqref{sum gw Dirac massive},
\begin{align*}
(1-|w|)&\approx|1-z^2|^{-\frac{1}{2}}(1+|z|)^{-1}\dist(z,\sigma(H_0)),\\
|w-w_1||w-w_2|&\approx|1-z^2|^{\frac{1}{2}}(1+|z|)^{-1},\\
|w-w_3||w-w_4|&\approx(1+|z|)^{-1}.
\end{align*}
These are easy consequences of the explicit formula \eqref{z=znu}; for the first, one uses Koebe's distortion theorem.
\end{proof}

\begin{theorem}\label{thm eigenvalue sums prseudorelativistic}
Let $d\geq 2$, $H_0=(1-\Delta)^{1/2}-1$, and let $V\in L^d(X)\cap L^{(d+1)/2}(X)$. Then, for any $\tau>0$, we have the estimate
\begin{align}\label{eigenvalue sums prseudorelativistic}
\sum_{z\in\sigma_{\rm d}(H_0+V)}\dist(z,\sigma(H_0))|z|^{\frac{d-1+\tau}{2}}(1+|z|)^{-\frac{3d-1}{2}-\tau}\leq C(V,\tau).
\end{align}
\end{theorem}

\begin{proof}
The conformal map $\psi$ in this case is the same as in \eqref{def. inverse conformal map fractional}, where we recall that $|z_0|>1$ depends on $V$ and is chosen in such a way that $z_0\in\rho(H_0)$.
By \eqref{log bound g}, with $N(z)$ as in Remark \ref{remark main theorem schatten norms}, we have that
\begin{align*}
\log|g(w)|&\leq CN(\psi^{-1}(w))^{\alpha}\|V\|^{\alpha}-\log|h(z_0)|\\
&\leq C\left(|\psi^{-1}(w)|^{-\frac{1}{2}}+|\psi^{-1}(w)|^{\frac{2d}{d+1}-1}\right)^{\alpha}\|V\|^{\alpha}-\log|h(z_0)|\\
&\leq C\left(|1+w|^{-1}+|1-w|^{-2\left(\frac{2d}{d+1}-1\right)}\right)^{\alpha}\|V\|^{\alpha}-\log|h(z_0)|\\
&\leq C(V)|1+w|^{-\alpha}|1-w|^{-2\alpha\frac{d-1}{d+1}}
\end{align*}
for $w\in\mathbb{D}$ and $\alpha=d+1$. 
Using \eqref{Borichev Golinski Kupin} with $\mathbf{b}=\{-1,1\}$ and $\mathbf{k}=\{d+1,2(d-1)\}$, it follows that for every $\tau>0$
\begin{align*}
\sum_{g(w)=0}(1-|w|)|1+w|^{d+\tau}|1-w|^{2d-3+\tau}\leq C(V,\tau).
\end{align*}
By the Koebe's distortion theorem and the explicit formulas \eqref{def. inverse conformal map fractional}--\eqref{def. conformal map fractional} for $\psi$ and $\psi^{-1}$ we obtain the following estimates,
\begin{align*}
(1-|w|)&\approx \frac{|z_0|^{1/2}}{|z|^{1/2}}\frac{\dist(z,[0,\infty))}{|z|+|z_0|},%\label{Koebe distortion 1}
\\
|1+w|&\approx \frac{|z|^{1/2}}{(|z|+|z_0|)^{1/2}},%\label{Koebe distortion 2}
\\
|1-w|&\approx \frac{|a|^{1/2}}{(|z|+|z_0|)^{1/2}}.% \label{Koebe distortion 3}.
\end{align*}
Plugging these estimates into the previous inequality and using $|z_0|>1$ yields the claim \eqref{eigenvalue sums prseudorelativistic}.
\end{proof}

\begin{remark}
Theorems \ref{thm eigenvalue sums big s}--\ref{thm eigenvalue sums prseudorelativistic} improve the results of \cite{Dubuisson2014,Dubuisson2014frac} in a similar way that \cite[Thm. 16]{FrankSabin14} improves those of \cite{DemuthEtAl2009,DemuthEtAl2013}. Notice that the bounds in \cite{Dubuisson2014,Dubuisson2014frac} depend (somewhat implicitly) on additional properties of $V$, not just its $L^q$-norm.
\end{remark}

\appendix
\section{Conclusion of the proof of Lemma \ref{lemma schatten estimate low frequency}}

It remains to prove that \eqref{inequality Schatten low frequency} holds for all $z\in\C^{\pm}$. This will be a slight modification to a standard argument involving the Phragm\'en-Lindel\"of maximum principle, see e.g.\ \cite[Sect. 5.3]{Ruiz2002}. 

\begin{proof}
We prove the claim in the case $2d/(d+1)\leq q\leq (d+1)/2$ and for $z\in\C^+$; an analogous argument can be used for $1\leq q\leq 2d/(d+1)$ or $z\in\C^-$. Consider the functions
\begin{align*}
f(z)=\Tr(|V|^{1/2}\chi^2(D)R_0(z)V^{1/2}F),\quad z\in\C^+,
\end{align*}
where $F$ is a finite rank operator %with $\|F\|_{\mathfrak{S}^{q(d-1)/(dq-d)}}=1$\footnote{$q(d-1)/(dq-d)$ is the dual exponent to $q(d-1)/(d-q)$.} and
with canonical representation
\begin{align*}
F=\sum_{i=1}^N\mu_i|u_i\rangle\langle v_i|.
\end{align*}
Here, $\{\mu_i\}_{i=1}^N$ are the singular numbers of $K$ and $\{u_i\}_{i=1}^N$, $\{v_i\}_{i=1}^N$ are orthonormal systems in $L^2(X)$. Since $\sgn(V)$ is unitary, we may assume that $V$ is positive. Moreover, by density, we may assume that $\mathcal{F}(V^{1/2})\in C^{\infty}_0(X)$. By the same argument, we may assume that
$\widehat{u_i},\widehat{v_i}\in C^{\infty}_0(X)$. We then prove the following.
\begin{itemize}
\item[i)] $f$ is analytic in $\C^+$ and has a continuous extension up to the boundary,
\item[ii)] $|f(\lambda)|\leq C\|V\|_{L^{q}}$, for $\lambda\in\R$,
\item[iii)] For every $\epsilon>0$ there is $C_{\epsilon}>0$ such that $|f(z)|\leq C_{\epsilon}\e^{\epsilon|z|}$ as $|z|\to\infty$ in $\C^+$, uniformly in the argument of $z$.
\end{itemize}
Proof of i): Complete $\{v_i\}_{i=1}^N$ to an orthonormal basis $\{v_i\}_{i=1}^{\infty}$ of $L^2(X)$. Evaluating the trace in this basis, we see that
\begin{align}\label{Fj of z as a finite sum}
f(z)=
\sum_{i=1}^N \mu_i\langle\chi(D)V^{1/2}v_i,\chi(D)R_0(z)V^{1/2}u_i\rangle.
\end{align}
This is an analytic function in $\C^+$ since the resolvent $R_0(z)$ is an analytic operator-valued function in this domain. From \eqref{distributional limit T-lambda} it follows that $f$ has a continuous extension to the boundary.
Indeed, combining \eqref{Fj of z as a finite sum} and \eqref{distributional limit T-lambda} and using Plancherel's theorem, we see that $f$ has a continuous extension to $\partial\C^+=\R$, given by
\begin{align*}
f(\lambda+\I 0)=
&\sum_{i=1}^N\mu_i \,\,p.v. \int_{X}\frac{\chi(\xi)^2}{T(\xi)-\lambda}\overline{\mathcal{F}(V^{1/2} v_i)(\xi)}\mathcal{F}(V^{1/2} u_i)(\xi)\rd\xi\\
&-\I\pi \sum_{i=1}^N\mu_i\int_{\{T(\xi)=\lambda\}}\frac{\chi(\xi)}{|\nabla T(\xi)|}\overline{\mathcal{F}(V^{1/2} v_i)(\xi)}\mathcal{F}(V^{1/2} u_i)(\xi)\rd\sigma(\xi).
\end{align*}

Proof of ii): By H\"older's inequality in Schatten spaces \cite[Thm. 2.8]{Simon2005} and \eqref{Schatten bound in proof for lambda pm i epsilon}
\begin{align*}
|f(\lambda)|&\leq \|V^{1/2}\chi^2(D)R_0(\lambda)V^{1/2}\|_{\mathfrak{S}^{q(d-1)/(d-q)}}\|F\|_{\mathfrak{S}^{(d-1)/(dq-d)}}\leq C\|V\|_{L^{2}},
\end{align*}
for $\lambda\in\R$ and for all finite rank operators $F$ with $\|F\|_{\mathfrak{S}^{(d-1)/(dq-d)}}=1$.

Proof of iii): A similar computation as above yields
\begin{align*}
|f(z)|\leq \sum_{i=1}^N \int_{X}\frac{\chi(\xi)^2}{|T(\xi)-z|}|\mathcal{F}(V^{1/2} v_i)(\xi)||\mathcal{F}(V^{1/2} u_i)(\xi)|\rd \xi
\end{align*}
Since 
\begin{align*}
\frac{1}{|T(\xi)-z|}\leq \frac{C}{|z|},\quad |z|\gg 1,\quad\xi\in\supp(\chi),
\end{align*}
we see that $|f(z)|\leq  C\e^{\epsilon|z|}$ for every $\epsilon>0$ when $|z|\gg 1$.

By the Phragm\'en-Lindel\"of maximum principle, i)--iii) imply that 
\begin{align}\label{Fjz bounded}
|f(z)|\leq C\|V\|_{L^{q}},\quad z\in\C^+.
\end{align}
By density of the finite rank operators in $\mathfrak{S}^{(d-1)/(dq-d)}$, the inequality \eqref{inequality Schatten low frequency} follows for all $z\in\C^+$.
\end{proof}

\section{A convexity lemma}

\begin{lemma}\label{lem: convexity}
Assume that $A$ is a bounded operator on $L^2(X)$ and that there exist $\alpha_1,\alpha_2\in [1,\infty]$ and $q_1,q_2\in [1,\infty]$ such that the following hold.
\begin{enumerate}
\item $|V|^{1/2}A|V|^{1/2}\in\mathfrak{S}^{\alpha_1}(L^2(X))$ for every $V\in L^{q_1}(X)$, and
\begin{align*}
\||V|^{1/2}A|V|^{1/2}\|_{\mathfrak{S}^{\alpha_1}}\leq C_1\|V\|_{L^{q_1}}.
\end{align*}
\item $|V|^{1/2}A|V|^{1/2}\in\mathfrak{S}^{\alpha_2}(L^2(X))$ for every $V\in L^{q_2}(X)$, and
\begin{align*}
\||V|^{1/2}A|V|^{1/2}\|_{\mathfrak{S}^{\alpha_2}}\leq C_2\|V\|_{L^{q_2}}.
\end{align*}
\end{enumerate} 
For $\theta\in[0,1]$, let $\alpha_{\theta}^{-1}=(1-\theta)\alpha_1^{-1}+\theta\alpha_2^{-1}$ and $q_{\theta}^{-1}=(1-\theta)q_1^{-1}+\theta q_2^{-1}$. 
Then $|V|^{1/2}A|V|^{1/2}\in\mathfrak{S}^{\alpha_{\theta}}(L^2(X))$ for every $V\in L^{q_{\theta}}(X)$, and
\begin{align*}
\||V|^{1/2}A|V|^{1/2}\|_{\mathfrak{S}^{\alpha_{\theta}}}\leq C_1^{1-\theta}C_2^{\theta}\|V\|_{L^{q_{\theta}}}.
\end{align*}
\end{lemma}

\begin{proof}
We fix $\theta\in [0,1]$ and $V\in L^{q_{\theta}}(X)\cap L^{\infty}(X)$ and consider the family of operators
\begin{align*}
T_z=|V|^{\frac{z}{2x_{\theta}}}A|V|^{\frac{z}{2x_{\theta}}},\quad q_1\leq \re \,z\leq q_2,
\end{align*}
where $x_{\theta}=(1-\theta)q_1+\theta q_2$. 
%It is easy to see that for any $f,g\in C_0^{\infty}(X)$, we have that
%\begin{align}\label{eq: T_z admissible}
%|\langle f,T_z g\rangle|\leq \|A\|\sup_{q_1\leq \re\, z\leq q_2}\||V|^{\frac{\re\, z}{2}}f\|_{L^2} \||V|^{\frac{\re\, z}{2}}g\|_{L^2},
%\end{align}
%where $\|A\|$ is the operator norm of $A$ on $L^2(X)$. 
Let $F$ be a finite rank operator on $L^2(X)$, and let $F=U|F|$ be its polar decomposition. We assume that $\|F\|_{\mathfrak{S}^{\alpha_{\theta}'}}=1$, where $\alpha_{\theta}'$ is the conjugate exponent to $\alpha_{\theta}$, i.e.\ $1/\alpha_{\theta}'+1/\alpha_{\theta}=1$. We consider the function
\begin{align*}
f(z)=\Tr(T_zU|F|^{l(z)}),\quad q_1\leq \re\,z\leq q_2, 
\end{align*}
where $l(z)$ is the linear function defined by
\begin{align*}
l(z)=\alpha_{\theta}'\left(\frac{q_2-z}{q_2-q_1}\frac{1}{\alpha_2'}+\frac{z-q_1}{q_2-q_1}\frac{1}{\alpha_1'}\right).
\end{align*}
Let $\{s_j^2\}_{j=1^N}$ be the singular values of $F$. Evaluating the trace with respect to a normalized eigenbasis $\{\phi_j\}_{j=1}^N$ of $F^*F$, we have that
\begin{align*}
f(z)=\sum_{j=1}^Ns_j^{l(z)}\langle|V|^{\frac{z}{2x_{\theta}}} \phi_j,A|V|^{\frac{z}{2x_{\theta}}} U\phi_j\rangle.
\end{align*}
Thus, $f:S\to\C$ is continuous on the strip $S=\set{z\in\C}{q_1\leq\re\,z\leq q_2}$ and analytic in its interior. It is easy to see that we have the estimate
\begin{align*}
|f(z)|\leq Ns_0^{\frac{\alpha_{\theta}'}{\alpha_1}}\|V\|_{L^{\infty}}^{\frac{q_2}{x_{\theta}}}\|A\|,\quad z\in S.
\end{align*}
Moreover, for $t\in\R$, by H\"older's inequality, we have that
\begin{align*}
|f(q_1+\I t)|&\leq \|T_{q_1+\I t}\|_{\mathfrak{S}^{\alpha_2}}\|F\|_{\mathfrak{S}^{\alpha_2'l(q_1)}}^{l(q_1)}\leq C_2\|V\|_{L^{\frac{q_1q_2}{x_{\theta}}}}^{\frac{q_1}{x_{\theta}}},\\
|f(q_2+\I t)|&\leq \|T_{q_2+\I t}\|_{\mathfrak{S}^{\alpha_1}}\|F\|_{\mathfrak{S}^{\alpha_1'l(q_2)}}^{l(q_2)}\leq C_1\|V\|_{L^{\frac{q_1q_2}{x_{\theta}}}}^{\frac{q_2}{x_{\theta}}}.
\end{align*}
Note that $q_1q_2/x_{\theta}=q_{\theta}$. By the three lines theorem, it follows that
\begin{align*}
|f(x_{\theta})|\leq C_2^{1-\theta}C_1^{\theta}\|V\|_{L^{q_{\theta}}}. 
\end{align*}
Since this inequality holds for any finite rank operator $F$ with
$\|F\|_{\mathfrak{S}^{\alpha_{\theta}'}}=1$, it follows that $T_{x_{\theta}}=|V|^{1/2}A|V|^{1/2}\in\mathfrak{S^{\alpha_{\theta}}}$ and 
\begin{align*}
\||V|^{1/2}A|V|^{1/2}\|_{\mathfrak{S}^{\alpha_{\theta}'}}\leq C_2^{1-\theta}C_1^{\theta}\|V\|_{L^{q_{\theta}}}.
\end{align*}
The restriction that $V$ is bounded can be removed by a density argument.
\end{proof}

\noindent
{\bf Acknowledgements.} 
{\small The author would like to thank Ari Laptev and Carlos E.\ Kenig for useful discussions. Moreover, the author is indebted to the referee for suggesting a stronger version of Lemma \ref{lemma schatten estimate low frequency}. Support of Schweizerischer Nationalfonds, SNF, through the postdoc stipends PBBEP2\_\_136596 and P300P2\_\_147746 is gratefully acknowledged.}

\bibliographystyle{plain}
\bibliography{bibliography}
\end{document}